\documentclass[journal,a4paper, 10pt, twocolumn]{article}

\usepackage[margin={1.2cm,2cm}]{geometry}
\title{Distributed recovery of jointly sparse signals\\ under communication constraints}   
\author{Sophie   M.  Fosson$^{\star   }$   \qquad   Javier Matamoros$^{\dagger}$ \qquad  Carles Ant\'{o}n-Haro$^{\dagger}$ \qquad Enrico  Magli$^{\star}$\\  $^{\star}$  Department of  Electronics  and
Telecommunications,     Politecnico    di     Torino     (Italy)    \\ $^{\dagger}$Centre  Tecnol\`{o}gic de Telecomunicacions  de Catalunya, Barcelona  (Spain)\thanks{This work is supported by the European Commission in the framework of the FP7 Network of Excellence in Wireless COMmunications, Grant agreement n.318306, by the European Research Council under FP7 / ERC, Grant agreement n.279848 - CRISP project, by the Spanish Government through the project  INTENSYV (TEC2013-44591-P), and by the Catalan Government (2014 SGR 1567). } 
}

\usepackage{algorithm, algorithmic}
\usepackage{graphicx} % for pdf, bitmapped graphics files
\usepackage{amsmath} % assumes amsmath package installed
\usepackage{amssymb}  % assumes amsmath package installed
\usepackage{amsthm}
\usepackage{amsfonts}
\usepackage{dsfont}
\usepackage{color}
\usepackage{subfig}
\usepackage{pgfplots}
\usepackage{tikz}
\usetikzlibrary{arrows}
\usetikzlibrary{intersections}
\usepackage[english]{babel}
\usepackage[latin1]{inputenc}

\newcommand{\cardV}{V}

\newcommand{\fun}{\mathcal{F}}

\newcommand{\sun}{\mathcal{R}}
\newcommand{\soft}{\mathbb{S}}
\newcommand{\softmod}{\mathbb{S}}

\newcommand{\V}{\mathcal{V}}

\newcommand{\NE}{\mathcal{N}}

\newcommand{\R}{\mathds{R}}
\newcommand{\N}{\mathds{N}}

\newcommand{\argmin}[1]{\underset{#1}{\mathrm{argmin\,}}}

\newcommand{\mw}{\overline{w_{v,i}}}
\newcommand{\mwt}{\overline{w_{v,i}}(t)}

\newtheorem{lemma}{Lemma}
\newtheorem{proposition}{Proposition}
\newtheorem{theorem}{Theorem}

\newtheorem{assumption}{Assumption}

\begin{document}

\maketitle

\begin{abstract}
The problem of the distributed recovery of jointly sparse signals has attracted much attention recently. Let us assume that the nodes of a network observe different sparse signals with common support; starting from linear, compressed measurements, and exploiting network  communication, each node aims at reconstructing the support and the non-zero values of its observed signal. In the literature, distributed greedy algorithms have been proposed to tackle this problem, among which the most reliable ones require a large amount of transmitted data, which barely adapts to realistic network communication constraints. In this work, we address the problem through a reweighted $\ell_1$ soft thresholding technique, in which the threshold is iteratively tuned based on the current estimate of the support. The proposed method adapts to constrained networks, as it requires only local communication among neighbors, and the transmitted messages are indices from a finite set. We analytically prove the convergence of the proposed algorithm and we show that it outperforms the state-of-the-art greedy methods in terms of balance between recovery accuracy and communication load.
\end{abstract}
%
%\begin{keywords} Joint  sparsity,  distributed algorithms, compressed sensing, iterative thresholding, reweighted $\ell_1$ minimization, concave penalization.
%\end{keywords} %--%
%\begin{multicols*}{2}
%\lipsum
%\end{multicols*}
\section{Introduction}\label{sec:intro}
The recovery of jointly sparse signals has received great attention in the last few years. By ``jointly sparse'' we mean signals that are sparse (\emph{i.e.}, have few non-zero components) with same support (\emph{i.e.}, the positions of the non-zero components are common for all the signals). Measurements of such signals are assumed to be taken by the nodes of a network;  given the measurements, the aim of each node is to estimate the common support  and eventually evaluate  the non-zero components. The study of this problem is motivated by diverse applications, among which one of the most outstanding is spectrum sensing in cognitive radio networks \cite{baz10, zen11}, which consists in the detection of the spectrum occupancy aimed to the dynamic reallocation of unused frequencies; as described in \cite[Section III.D]{zen11}, in some cases this problem reduces to the reconstruction of a common support. Other examples of jointly sparse representations, just to name a few among the most recent ones, arise from 
image 
features extraction \cite{nan11}, visual classification \cite{yua12}, speech recognition \cite{wei13}, and biometrics recognition \cite{she14}.

In several applications, measurements are linearly acquired and compressed \cite{baz10,zen11}, according to the distributed compressed sensing (CS) paradigm \cite{bar05, dua05}.  CS \cite{don06} states that a sparse signal $x\in\R^n$ can be recovered from measurements $y=Ax$ where $A\in\R^{m,n}$ is a suitable matrix with $m<n$, called sensing matrix. In a distributed context, the acquisition is performed by a networked system: given a set $\V$ of nodes, each $v\in\V$ has its own measurement $y_v=A_v x_v$; the case when the $x_v$'s have common support is known as joint sparsity model 2 (JSM-2, \cite{dua05}). Concerning the recovery methods, centralized and distributed methods have to be distinguished. The first ones assume the presence of a fusion center that  gathers all the information from the network (namely, measurements and sensing matrices) and processes them to recover the signals. In the case that all the sensing matrices are equal, these methods can be recast in the multiple measurement 
vectors framework (MMV) \cite{che06}, for which theoretical recovery guarantees have been provided \cite{che06, dav12}. More insight on the recovery methods for MMV can be found in very recent papers such as \cite{kyr12,bla14}. The distributed recovery methods, instead, perform the reconstruction in-network, with no fusion center,  only exploiting the computational and (local) communication capabilities of the nodes. Distributed methods are remarkable as (a) they do not need the presence of a fusion center, which in many situations is not available or can be expensive to reach in terms of transmit power (sensor networks are often deployed over impracticable  territories for environment monitoring purposes); (b) they are more robust to failures: if a fusion center breaks down, the recovery process stops, while if a distributed algorithm is run in-network, typically the failure of some nodes is tolerated.

The development of distributed recovery algorithms for JSM-2 is  our purpose. The literature on this argument is very recent. First attempts \cite{lin11}, \cite[Section III.D]{zen11} went in the direction of decentralizing group Lasso techniques \cite{yua06}, but no convergence guarantees were provided. Distributed greedy algorithms  were then studied: in \cite{sun14}, distributed versions of subspace pursuit (SP) and orthogonal matching pursuit (OMP) were developed, the second one (called DiOMP) being more promising in terms of recovery performance. The support recovery accuracy of DiOMP is comparable to that of DiT in \cite{fox14}, which is the first distributed algorithm based on iterative thresholding for JSM-2.  Almost at the same time,  in \cite{wima14} DC-OMP 1 was proposed, which is very similar to DiOMP, but more accurate in the support detection. A second algorithm was proposed in \cite{wima14}, named DC-OMP 2, which recovers the support much more accurately than DC-OMP 1,  at the price of a greater communication load. To the best of our knowledge, DC-OMP 1 and DC-OMP 2 represent the state of the art in the framework of distributed algorithms for JSM-2 and will be considered as benchmark in this work; in the following, we will describe them more in detail.
%%zen11 usa poi ADMM. Non l'ho detto da nessuna parte, ma per ora non lo dico
The aim of this paper is to present a new approach to the distributed recovery of jointly sparse signals, based on concave penalization and reweighted $\ell_1$ minimization. More precisely, we will develop a distributed soft thresholding in which the threshold is iteratively updated,  based on the support estimate. With our method,  communication can be strongly reduced with respect to DC-OMP 2 (with no performance loss), being limited to the local communication of the indices of the components that have switched from non-zero to zero or vice versa. In other terms, our algorithm will be efficient even under strict communication constraints, due to the network technology or for energy saving purposes. Our algorithm will be proved to converge to a minimum of suitable cost functional, and performance will  be shown via numerical simulations. 

% Another advantage of soft thresholding with respect to  is that no prior knowledge about the sparsity level is required, which instead  is necessary for greedy algorithms. Our algorithm will be rigorously developed starting from a suitable cost functional and finding an iterative method to decrease it. From this descent we will derive the proof of convergence for regular 
% networks. 

The paper is organized as follows.  In Section \ref{sec:model}, we will describe the model, and in Section \ref{sec:optimization} we will establish our optimization problem. In Section \ref{sec:algorithm}, we will present and discuss our algorithm. In Section \ref{sec:asymp_reg} we will prove the numerical convergence and the stabilization of the support estimate, while  the convergence of  the non-zero components will be discussed in Section \ref{sec:convergence}. Numerical results will be then shown in Section \ref{sec:numerical_results}, along with an analysis of the transmission costs. Finally, some conclusions will be drawn.
  
Before proceeding, we anticipate some notation that will be used throughout the paper.
\subsection{Notation}\label{par:notation}
We denote by $\mathds{1}$ the indicator function: for any integer $n\geq 1$,  $\mathds{1}:\R^n\mapsto\R^n$ is given by  $[\mathds{1}(x)]_i=1$ if $x_i\neq 0$, while  $[\mathds{1}(x)]_i=0$ if $x_i= 0$, $i=1,\dots,n$. $\mathbf{1}$ indicates the column vector whose components are all equal to 1.
We define the $l_0$-norm of a vector $x\in\R^n$ as $\left\|x\right\|_0=\left\|\mathds{1}(x)\right\|_2^2$, or equivalently  $\left\|x\right\|_0=\mathbf{1}^{\mathsf{T}}\mathds{1}(x)$, where $\mathsf{T}$ indicates the transpose. $I$ is the identity matrix. Moreover, we  call weighted $l_p$-norm of $x$ the quantity $\left\|W x\right\|_p$ where $W$ is a  weight matrix, namely a diagonal matrix with diagonal entries $W_i>0$, $i=1,\dots,n$. Given a graph  $\mathcal{G}=\left(\mathcal{V},\mathcal{E}\right)$, for any node $v\in\V$,  $\mathcal{N}_v:=\{w\in\V \text{ s.t. } (v,w)\in\mathcal{E}\}$ is the neighborhood of $v$. Let $d_v$ be the degree of $v$, say  the number of neighbors of $v$, included $v$ itself. Given any variable $x_v$ associated with $v$, we   indicate its local average with an overline: $\overline{x}_v:=\frac{1}{d_v}\sum_{w\in\mathcal{N}_w}x_w$  (we remark that $:=$ denotes ``is defined as'').
 %--%
\section{Network model}\label{sec:model}
In this section, we describe the acquisition and communication model of interest.

We consider  a network composed  of $V$  nodes, whose  connectivity is described  by  the  graph   $\mathcal{G}=\left(\mathcal{V},
\mathcal{E}\right)$  with  $|\mathcal{V}|=V$.  Accordingly,  the node $v$  can communicate  with  $v'$ if and  only if $\{v,v'\}\in \mathcal{E}$ or, in  other words, if $v'$ belongs to its neighborhood set $\mathcal{N}_{v}$.

Following the CS paradigm, each node  observes a compressed version of  a $k$-sparse signal $\{x^{\star}_v\}_{v\in\mathcal{V}}  \in  \mathds{R}^n$  through a  set  of linear measurements, namely
\begin{equation}\label{model}
\begin{split} &y_v=A_v x^{\star}_v,\quad v\in \mathcal{V}
\end{split}
\end{equation}
where $A_v\in\R^{m\times n}$ (with $m < n$)  and the signals $\{x_v\}_{v\in\mathcal{V}}$ have  the   same  support $\Omega$,   that  is, for all  $v\in\V$,  $ \Omega_v:=\left\{i\in\{1,\dots,n\}|x^{\star}_{v,i}\neq 0\right\}=\Omega$.  In the next, we will equivalently refer to the support of $x_v$ as the binary vector $\mathds{1}(x_v)$. 

A measurement noise term can be added in \eqref{model} to have a more realistic setting.  If we assume an additive white Gaussian noise (a popular choice in a number of  applications), the formulation and the approach to the problem do not change with respect to the noiseless case, as we consider the  least squares paradigm, which in both cases considers the minimization of the residual.

The  ultimate  goal  of  each  node $v\in\V$ is the  reconstruction of its observed signal $x_v$. A fusion center is not envisaged in our model, thus the reconstruction task  has to be performed in-network by the nodes themselves. Moreover, we assume that no information about $A_v$ and $y_v$ can be shared, \emph{e.g.}, for privacy reasons and to reduce the amount of transmitted data. Since the transmission load is often a dramatic drawback in distributed procedures, we impose a second constraint on the communication protocol: messages must belong to a finite set of integers, specifically $\{1,\dots,n\}$. This should adapt to our purpose: since the support is the common quantity, it should be sufficient to share information about the support of each component, which is a binary message. In other terms, for each component $i$ a node would communicate its status, that is, if in its current estimate $i$ is in the support or not; assuming that the other nodes can store such information, it is sufficient to send the 
value $i$ when the status has changed. For each sent message, we then need only  $\lfloor \log_2 n \rfloor +1$ bits, which generally is significantly smaller than the number of bits used to transmit a real number, even if coarsely quantized. 

Let us summarize these communication constraints.

\begin{assumption}\label{comm_constraints}
The communication over the network is local, and only messages in $\{1,\dots,n\}$  can be transmitted by each node to the neighbors.
\end{assumption}

It is well known that, in the CS context, the challenge is the identification of the signal support; once this is done, the estimate of the non-zero components could be readily performed through the classical least squares estimation (assumed the number of measurements is larger than the sparsity). For this motivation, in the literature \cite{sun14, wima14} the detection of the signal support is approached separately. Our proposed method instead will envisage both support and non-zero values recovery in the same algorithm.

%Let us assume that  $\mathcal{G}$ is a regular graph with degree $d$ (this simplifies the notation and the following convergence analysis). 
\section{Optimization problem}\label{sec:optimization}
Given the network model presented in Section \ref{sec:model}, we now describe our recovery problem in terms of an optimization problem, that takes into account the network constraints of Assumption \ref{comm_constraints}.  Our final purpose is the development of a distributed recovery algorithm that leverages iterated sharing of information about the support.

In the context of sparse recovery, the  $\ell_1$ convex minimization problem, known as Lasso, is very popular for its mathematical feasibility. The principle behind Lasso is that $\ell_1$ norm well approximates the $\ell_0$ norm and allows to transform the recovery problem into a convex problem.  Further, \emph{reweighted $\ell_1$ minimization} \cite{can08rew, wip10, che14} has been proposed, which iteratively retunes the weight of the $\ell_1$ norm based on the current signal's estimate. In this way, each component is weighted according to its expectation of belonging to the support. Different reweighting rules have been investigated in the literature, and will be discussed later. 

The reweighting principle seems to be suitable for distributed support detection: intuitively we can think of an $\ell_1$-reweighting minimization at each node, in which the reweighting rule depends on the individual current estimate and on the support information shared in the network. In other terms, we aim for a decentralization of reweighted $\ell_1$ minimization.

The rest of the section is devoted to develop this idea. We start with a review on (centralized) \emph{concave penalization}, which is the setting where the $\ell_1$ reweighting techniques are originated. Afterwards, we will illustrate how to decentralize this method, taking into account our model constraints (Assumption \ref{comm_constraints}).
\subsection{From Lasso to concave penalization}
As mentioned before, the problem of sparse signals' recovery  can be conceived as an $\ell_1$ convex minimization problem, known as Lasso: 
\begin{equation}\label{lasso}  
 \min_{x\in\R^n}\frac{1}{2}\left\|y-A x \right\|_2^2+\lambda \| x\|_1,~~~\lambda>0
\end{equation}
where $A\in\R^{m\times n}$, and $\lambda$ is a parameter to set.  As already said, the $\ell_1$ norm has been shown to well approximate the $\ell_0$ norm, and has the great advantage of transforming the problem from combinatorial to convex. However, Lasso has some drawbacks, namely its estimate is always biased (proportionally to $\lambda$), and conditions to have the oracle property (\emph{i.e.}, the capability of exactly recovering the support) are strict \cite{fan01_pioneer, zha06, wai09}. This has motivated the studies on different penalization techniques. In particular, much interest has been devoted to concave penalization techniques:
\begin{equation}\label{concave_penalization}  
\begin{split}
 &\min_{x\in\R^n}\frac{1}{2}\left\|y-A x \right\|_2^2+\lambda \sum_{i=1}^n g(|x_i|)\\&g:\R_+\to \R_+ \text{ concave, nondecreasing in } |x_i|.
\end{split}
\end{equation}

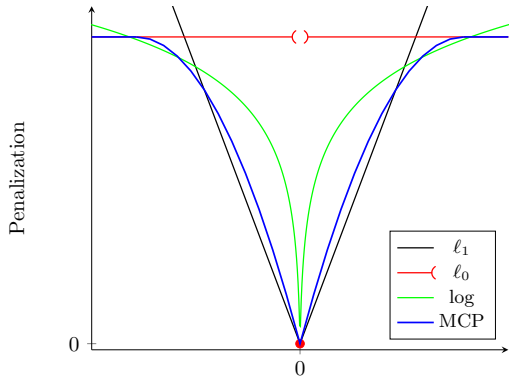
\begin{figure}
 	\begin{tikzpicture}[scale=0.8]
		\begin{axis}[xmin=-1.8,xmax=1.8,ymin=-0.02, ymax=1.1, xtick={0}, ytick={0}, ylabel={Penalization},axis lines=left, legend entries={\small{$\ell_1$}, \small{$\ell_0$} , ,\small{$\log$}, \small{MCP }},legend pos= south east]    
		\addplot[name path = funone, draw=black, mark=none, semithick] {abs(x)};
			\addplot[draw=red, mark=none, semithick, domain=-1.8:-0.05,-(,shorten >=-1pt] 
		{1};
			\addplot[ draw=red, mark=none, semithick, domain=1.8:0.05,-(,shorten >=-1pt] 
		{1};
		\draw[color=red, fill] (axis cs:0,0) circle (2pt);
 			\addplot[draw=green, mark=none, semithick,samples=1500] {0.2*ln( 100*abs(x)+1)}; 
 			\addplot[draw=blue, mark=none,thick,domain=-sqrt(2):sqrt(2)]{abs(x)*sqrt(2)-0.5*x^2} ; 
 			\addplot[ draw=blue, mark=none, thick, domain=-1.8:-sqrt(2)] {1};
 			\addplot[ draw=blue, mark=none, thick, domain=1.8:sqrt(2)] {1}	;
\end{axis}
	\end{tikzpicture}
	\caption{Examples of popular concave penalization functions, that are closer to $\ell_0$ than $\ell_1$. In this work, we focus on  MCP.}
	\label{fig:penalty}
\end{figure}
The rationale behind this is that concave functions approximate the $\ell_0$ norm  better than $\ell_1$, as one can appreciate in Figure \ref{fig:penalty}. Many contributions on concave penalization come from the statistical community, see, \emph{e.g.}, \cite{fan01_pioneer, fan11, fan14, zou08_LLA, zha10MCP, zha12, gas09}. In such papers, different concave $g$'s have been proposed, and conditions to have the oracle property and to reduce the Lasso bias have been studied, mainly in the asymptotic  case $n\to\infty$ \cite{fan01_pioneer, fan11}. Experimental and theoretical results attest that usually concave penalization outperforms Lasso \cite{fan01_pioneer, zha10MCP, zha12}. In the context of underdetermined linear systems, some works \cite{can08rew, faz03} apply the concave penalization to CS and matrix rank minimization with success.

The concave penalization problem \eqref{concave_penalization} is not mathematically straightforward: non-convexity makes it difficult to find  global solutions. However,  in many cases local minima are precise enough, and  can be reached via iterative methods based 
 on \emph{linear local approximation} (LLA) of $g$  \cite{zou08_LLA, can08rew, faz03}. Given a point $z_i\in\R_+$, the key idea of LLA is to substitute $g(|x_i|)$ around $z_i$ by its linearization $g(z_i)+g'(z_i)(|x_i|-z_i)$; thanks to concavity, $g$ is always below its linearization, which suggests the following procedure. Assuming that $z$ is the current estimate, we locally minimize \eqref{concave_penalization} substituting $g$ with its linearization. Removing the constant terms, we obtain:
\begin{equation}\label{concave_penalization_lla}  
 \min_{x\in\R^n}\frac{1}{2}\left\|y-A x \right\|_2^2+\lambda \sum_{i=1}^n g'(|z_i|)|x_i|.
\end{equation}
Let us suppose that  an estimate $z=x(t)$ is provided at current time $t\in\N$. Then, we can perform alternated minimization on \eqref{concave_penalization_lla}: 
\begin{equation}\label{algo:reweighting}
 \begin{split}
 &x(t+1)= \min_{x\in\R^n}\frac{1}{2}\left\|y-A x \right\|_2^2+\lambda \sum_{i=1}^n w_i(t)|x_i|\\
 &w_i(t+1)= g'(|x_i(t+1)|).\\
 \end{split}
 \end{equation}
This turns out to be is an iterative  reweighted $\ell_1$ minimization procedure. Such method  has been proved to reach a local minimum of the concave penalization functional, and in practice it is more accurate than Lasso global solution \cite{zou08_LLA, can08rew, faz03}. We remark  that no general guarantee of convergence for $x_i(t)$ is provided, but specific results hold for specific $g$'s. For example in \cite{che14}, convergence is proved for $g(|x_i|)=(|x_i|+\epsilon)^p$, with $p\in (0,1)$ and small $\epsilon>0$. 
 
In the literature, a variety of concave penalization functions have been investigated.  In \cite{can08rew} much attention is focused on the case $g(|x_i|)=\log(|x_i|+\epsilon)$, with small $\epsilon>0$. SCAD \cite{fan01_pioneer} and MCP \cite{zha10MCP} instead propose continuous quadratic penalizations: MCP is of the form $g(x_i)=\alpha|x_i|-\beta x_i^2$,  for $|x_i|< \frac{\alpha}{2\beta}$, $\alpha, \beta>0$, and constant otherwise; SCAD is like MCP plus a $\ell_1$ penalization term $\lambda|x_i|$  for small $|x_i|$. In the cited works, in-depth analyses and comparisons between the different $g$'s are proposed.

In conclusion, concave penalization provides us (a) a sparse recovery setting that outperforms Lasso, and (b) low complex algorithms, based on LLA, to find a solution. The LLA algorithms are nothing but reweighted $\ell_1$ schemes.
\subsection{Decentralization under communication constraints}
Our aim is to decentralize the problem \eqref{concave_penalization} and the algorithm sketched by \eqref{concave_penalization_lla}-\eqref{algo:reweighting} under communication constraints (Assumption 1). First of all, we notice that the natural way to write the optimization problem over the network is the summation of the individual functionals \eqref{concave_penalization} for each node $v\in\V$. Second, we observe that  the penalization is strictly linked to the support: as explained in \cite{can08rew}, in \eqref{concave_penalization_lla} we would desire larger $w_i$'s for the zero components, up to the ideal case when $w_i\to\infty$ for zero components, and $w_i\to 0$ for non-zero components. Since here signals have common support, it makes sense to compute $g$ over a common variable of the network, and the simplest choice is the mean. Summing up, we have:
\begin{equation*}%\label{my_lasso_distributed}  
\begin{split}
\min_{x_v\in\R^n}&\sum_{v\in\V}\left\{\frac{1}{2}\left\|y_v-A_v x_v \right\|_2^2+\lambda \sum_{i=1}^n g\left(\frac{1}{\cardV}\sum_{v\in\V} |x_{v,i}|\right)\right\}.
\end{split}
\end{equation*}
Nevertheless, this would require global communication to update  $w$ in the procedure \eqref{algo:reweighting}, which is in contrast with Assumption \ref{comm_constraints} for non-complete graphs. We then use the best local approximation that we can conceive, that is, we substitute  $\frac{1}{\cardV}\sum_{v\in\V}|x_{v,i}|$ with the local sum $\frac{1}{|\NE_v|}\sum_{u\in\NE_v}|x_{u,i}|$. The corresponding functional is
\begin{equation*}%\label{my_lasso_distributed_local}  
\begin{split}
\min_{x_v\in\R^n}&\sum_{v\in\V}\left\{\hspace{-0.07cm}\frac{1}{2}\left\|y_v-A_v x_v \right\|_2^2+\lambda \sum_{i=1}^n g\left(\frac{1}{|\NE_v|}\sum_{u\in\NE_v}|x_{u,i}|\right)\hspace{-0.1cm}\right\}\hspace{-0.05cm}.
\end{split}
\end{equation*}
In this way, each $v\in\V$ will have its own weight $w_v$, which will be reweighted using only local collaboration.
%We remark that even though the $w_v$'s are different, some consensus is expected when information has propagated over the network.

%We then have $w_{v,i}(t)=\beta - \frac{1}{|\NE_v|}\sum_{\in\NE_v}|x_{u,i}(t)|$ (or zero). 

Finally,  according to Assumption \ref{comm_constraints} the transmission of real valued messages (such as $|x_{u,i}|$) is undesired. Therefore, we impose that each node $v$ cannot access $x_u$, $u\in\NE_v\setminus\{v\}$, but only their best ``binary approximation'', say  $\mathds{1}(x_{u,i}(t))$. We then substitute $|x_{u,i}|$ by  $\mathds{1}(x_{u,i})$, and obtain our ultimate minimization problem: given $X=(x_1,\dots,x_{\cardV})$, we write
\begin{equation}\label{ultimate_problem}
\min_{x_v\in\R^n}\fun (X)
\end{equation}
where
\begin{equation*}%label{my_lasso_distributed_local_binary}  
\begin{split}
%&\min_{x_v\in\R^n}\fun (X)\\
&\fun(X)=\hspace{-0.08cm}\sum_{v\in\V}\left\{\hspace{-0.08cm}\frac{1}{2}\left\|y_v-A_v x_v \right\|_2^2+\lambda \sum_{i=1}^n g\left(\alpha | x_{v,i}| +\overline{\mathds{1}(x_{v,i}})\right)\hspace{-0.08cm}\right\}\\
\end{split}
\end{equation*}
and  $\overline{\mathds{1}(x_{v,i})}=\frac{1}{|\NE_v\setminus{v}|}\sum_{u\in\NE_v\setminus{v}}\mathds{1}(x_{u,i})$\footnote{We remark that $\mathds{1}(x_{u,i})$ is a function of $|x_{u,i}|$, which guarantees that the current $g=g\left(\alpha | x_{v,i}| +\overline{\mathds{1}(x_{v,i}})\right)$ is still a function of the absolute values.}.  $\alpha>0$ is a tuning parameter: since we are summing quantities that are physically different (a magnitude $|x_{v,i}|$ and binary information), it could be useful to balance their contributions, \emph{e.g.} based on prior information on the energy of the signal. In practice, we have noticed that if each $v$ adds also  $\mathds{1}(x_{v,i})$, performance improves; therefore, in the following  we will use $\overline{\mathds{1}(x_{v,i})}=\frac{1}{|\NE_v|}\sum_{u\in\NE_v}(\mathds{1}(x_{u,i}))$.

Summing up, the LLA procedure applied to $\fun(X)$ originates the following decentralized reweighted $\ell_1$ minimization procedure: 
\begin{equation}\label{algo:my_reweighting}
 \begin{split}
 &x_v(t+1)= \min_{x_v\in\R^n}\fun_w(X)\\
 &w_{v,i}(t+1)= g'\left(\alpha | x_{v,i}(t+1)| +\overline{\mathds{1}(x_{v,i}(t+1))}\right)\\
 \end{split}
 \end{equation}
 where $\fun_w(X)$ is $\fun(X)$ with $ w_{v,i}(t)\left[\alpha |x_{v,i}|+\overline{\mathds{1}(x_{v,i})} \right]$ instead of $g\left(\alpha | x_{v,i}| +\overline{\mathds{1}(x_{v,i}})\right)$.
 %$$\fun_w(X)=\hspace{-0.08cm}\sum_{v\in\V}\left\{\hspace{-0.08cm}\frac{1}{2}\left\|y-A_v x_v \right\|_2^2+\lambda \sum_{i=1}^n w_{v,i}(t)\left[\alpha |x_{v,i}|+\overline{\mathds{1}(x_{v,i})} \right]\hspace{-0.08cm}\right\}.$$

Assuming that each $v$ can store $n$ bits for each one of its neighbors, the neighbors are just required to broadcast the message $i$ when the status (0 or 1) of the component $i$ has changed in the current estimation, which fulfills Assumption \ref{comm_constraints}.

%In conclusion, applying LLA to minimize $\fun(X)$, we obtain a reweighting scheme that is compliant with the prescribed communication constraints. 
Concerning the update of $x_v$ in \eqref{algo:my_reweighting}, three tricky points arise and will be discussed in next section. The minimization of $\fun_w(X)$ over $x_v$:
\begin{enumerate}
\item  is not a classical Lasso minimization due to the presence of the terms $\overline{\mathds{1}(x_v)}$;
\item requires the local communication of the $w_v$'s, which is still in contrast with Assumption \ref{comm_constraints};
\item is too fast for our networked problem: we observed in fact that the whole procedure converges after few iterations. This is undesirable because it does not allow propagation of the information over the network. We will then make the procedure slower by not computing the minimum, but just decreasing $\fun$ with respect to $x_v$, via an iterative  thresholding step.
\end{enumerate}

Before proceeding, we specify that in this work we will focus on the following concave penalization function $g$: 
\begin{equation}\label{my_g}
g(|z|)=\left\{\begin{array}{lr}
\beta |z|-\frac{1}{2}z^2&\text{ if } 0\leq z<\beta\\
\frac{1}{2}\beta^2 &\text{ otherwise.}\\
\end{array}\right.~~z\in\R,\beta>0
\end{equation}
This $g$ belongs to the family of MCP penalization functions \cite{zha10MCP}, and has been recently exploited in applications such as wavelets \cite[Equation 2.8]{ant11} and Gaussian Bayesian networks \cite{ara15}. As explained in \cite{zha10MCP}, MCP is appreciated as it minimizes the maximum concavity. In Figure \ref{fig:penalty} we compare  $g$ in \eqref{my_g} to other classical choices. Notice that when $|z|\geq \beta$, $g$ is constant and more penalization is applied, hence $\beta$ is a penalization threshold that can be tuned based on the problem. With \eqref{my_g},  in \eqref{algo:my_reweighting} we have: 
% \begin{equation}\label{my_final_w}
% w_i(t)=\left\{\begin{array}{lr}
% \beta -\alpha | x_{v,i}(t)| -\overline{\mathds{1}(x_{v,i}}(t) &\text{ if } \alpha | x_{v,i}(t)| +\overline{\mathds{1}(x_{v,i}}(t)<\beta\\
% 0 &\text{ otherwise.}\\
% \end{array}\right. 
% \end{equation}
\begin{equation}\label{my_final_w}
w_{v,i}(t)=[\beta -\alpha | x_{v,i}(t)| -\overline{\mathds{1}(x_{v,i}(t)}]_+ 
\end{equation}
where $[z]_+=\max\{0,z\}$, $z\in\R$.

The motivation to focus on \eqref{my_g} is twofold: on one hand, experimental results are satisfactory (see Section \ref{sec:numerical_results}); on the other hand, the mathematical simplicity of \eqref{my_g}  allows us to provide a complete convergence analysis of $x_v(t)$ (see Section \ref{sec:convergence}). In the next section, we discuss the update of $x_v(t)$ using this $g$, and we finally state our algorithm. 

\section{Proposed algorithm}\label{sec:algorithm}
Let us tackle points 1), 2), and 3) underlined in the previous section, that complicate algorithm \eqref{algo:my_reweighting}. First of all, let us notice that we can separate the terms of $\fun_w(X)$ that depend on single $x_v$'s, and we indicate them by $\fun_w(x_v)$:
\begin{equation}
\begin{split}
\fun_w(x_v)=&\frac{1}{2}\left\|y-A_v x_v \right\|_2^2+\lambda \sum_{i=1}^n w_{v,i}\alpha |x_{v,i}|\\&+\lambda\sum_{i=1}^n  \mathds{1}(x_{v,i}) \sum_{u\in\NE_v} \frac{w_{v,i}}{|\NE_u|}. 
\end{split}
\end{equation}

This formula highlights that each $v\in\V$ has to solve a Lasso with an extra term, \emph{i.e.}, a weighted $\ell_0$ norm, as anticipated in point 1), Section \ref{sec:optimization}.
In other terms, $\fun_w(x_v)$ has both $\ell_1$ and $\ell_0$ penalizations. Moreover, point 2) is now evident: the transmission of the neighboring $w_u$'s is necessary to compute $\mathds{1}(x_{v,i}) \sum_{u\in\NE_v} \frac{w_{v,i}}{|\NE_u|}$. In the next, we will use the  notation $\mw=\mathds{1}(x_{v,i}) \sum_{u\in\NE_v} \frac{w_{v,i}}{|\NE_u|}.$

In order to face point 3), we replace the minimization step with a decreasing step, that slows down the algorithm's convergence. Given the shape of $\fun_w(x_v)$, iterative thresholding is a suitable choice for this purpose. In \cite[Section 4.1]{for10}, the \emph{soft} thresholding algorithm has been proved to decrease the Lasso functional \cite[Lemma 4.3]{for10} by showing that it iteratively minimizes a properly augmented functional, known as surrogate functional. A similar property has been proved also for the \emph{hard} thresholding algorithm in \cite{blu08}, which decreases the $\ell_0$ penalized functional. Here, we use the same scheme based on the surrogate functional to develop an iterative thresholding algorithm that decreases  $\fun$. Due to the presence of both $\ell_1$ and $\ell_0$ terms, such procedure will merge soft and hard features. We refer the interested reader to \cite{fou11, kyr11} and to \cite{kyr12} for a deeper insight into hard and soft/hard thresholding techniques, respectively.
 
 We remark that efficient methods like the alternating direction method of multipliers (ADMM), \cite{boy10,yan11} cannot be directly implemented  due to the non-convexity of $\fun$. This will be further elaborated in Sections \ref{sub:why_not_admm} and \ref{sub:ADMM}. On the other hand, in the literature algorithms for the minimization of non-convex, non-smooth problems have been recently presented \cite{bol14, bag13, che12, bur05}, which here cannot be applied due to the non-continuity of $\fun$.

Let $B=(b_1,\dots, b_{\cardV})\in\R^{n\times \cardV}$. We define the surrogate functional as follows (see  \cite[Section 4.1.1]{for10} and \cite[Section 2.2]{blu08}):
\begin{equation*}
 \begin{split}
\sun(X,B)\hspace{-0.1cm}:=&\fun(X)\hspace{-0.1cm}+\hspace{-0.1cm}\frac{1}{2}\sum_{v\in\V}\left[\frac{1}{\tau} \left\|x_v-b_v \right\|_2^2-\left\|A_v (x_v-b_v) \right\|_2^2\right].
\end{split}
\end{equation*}
By defining $z_{v}:=b_{v}+\tau A_v^{\mathsf{T}}(y_v-A_v b_v)$, the following equality can be readily proved (\cite[Section 4.1.1]{for10}):
\begin{align*}
&\left\|y_v-A_v x_v \right\|_2^2+ \frac{1}{\tau}\left\|x_v-b_v \right\|_2^2-\left\|A_v(x_v-b_v) \right\|_2^2=\\
&~~= \frac{1}{\tau}\left\|x_v-z_v\right\|_2^2+\mathsf{const}
\end{align*}
where $\mathsf{const}$ is a term not depending on $x_v$. Hence, we can write the surrogate of each $\fun_{w}(x_v)$ as:
\begin{equation}\label{fun_dep_xvi}
\sun_w(x_{v,i})=\frac{1}{2\tau} (x_{v,i}-z_{v,i})^2+ \lambda\left[ \alpha w_{v,i} | x_{v,i}| +\mathds{1}(x_{v,i})\mw\right].
\end{equation}
Following the procedure in \cite[Section 4.1.1]{for10}, we minimize $\sun_w(x_{v,i})$ in \eqref{fun_dep_xvi} with respect to $x_{v,i}$. We distinguish two cases.

\begin{enumerate}
 \item $|z_{v,i}|\leq w_{v,i}$:  $\argmin{} \sun(x_{v,i})=0$.
 
 In fact, if $|z_{v,i}|\leq w_{v,i}$ and $x_{v,i}\neq 0$, the derivative of $\sun_w(x_{v,i})$ is $x_{v,i}-z_{v,i}+\text{sgn}(x_{v,i})w$, which is positive for $x_{v,i}>0$, and symmetrically negative for $x_{v,i}<0$. We then have the infimum points 
 $\lim_{x_{v,i}\to 0+}\sun_{x_{v,i}}=\frac{1}{2\tau}z_{v,i}^2+\lambda\mw\geq \frac{1}{2\tau}z_{v,i}^2= \sun_w(0)$, which shows that the global minimum is in zero, as depicted in Figure \ref{fig:1}.(a). 
 
 \item $|z_{v,i}|> w_{v,i}$: if $(|z_{v,i}|-w_{v,i})^2 < 2\tau\lambda\mw $, $\argmin{} \sun(x_{v,i})=z_{v,i}-w_{v,i}\text{sgn}(x_{v,i})$; otherwise, $\argmin{} \sun(x_{v,i})=0$.

 In fact, if $|z_{v,i}|> w_{v,i}$ and $x_{v,i}\neq 0$, the derivative of $\sun_w(x_{v,i})$ is zero (and we have a minimum) for $x_{v,i}=z_{v,i}-w_{v,i}\text{sgn}(x_{v,i})$, that is, $x_{v,i}=z_{v,i}-w_{v,i}$ if $z_{v,i}> w_{v,i}$, and $x_{v,i}=z_{v,i}+w_{v,i}$ if $z_{v,i}<- w_{v,i}$. This is not sufficient: this minimum has to be compared with $\sun_w(0)$, which, due to discontinuity, should be lower (see Figure \ref{fig:1}.(b)-(c)) This  occurs for $(|z_{v,i}|-w_{v,i})^2 < \tau\lambda\mw $, since
 $\sun_w(z_{v,i}-w_{v,i}\text{sgn}(x_{v,i}))=\frac{1}{2\tau}w_{v,i}(2|z_{v,i}|-w_{v,i})$ and $\sun_w(0)=\frac{1}{2\tau}z_{v,i}^2$.
 
 \end{enumerate}

 We observe that, despite the discontinuity in zero, the case $|z_{v,i}|\leq w_{v,i}$ is  analogous to soft thresholding. That is, the presence of the $\mathds{1}(x_{v})$ term does not change the position of the minimum (Figure \ref{fig:1}.(a)). However, when $|z_{v,i}|> w_{v,i}$  the term $\mathds{1}(x_{v})$ induces to choose zero more often than soft thresholding.

\begin{figure*}
\centering
\subfloat[$|z_{v,i}|<w_{v,i}$]{
\begin{tikzpicture}[scale=0.8]
\begin{axis}[xmin=-4,xmax=4,ymin=-14, ymax=13, xtick={0}, ytick={-12}, xticklabels={0},yticklabels={,,}] 
    \addplot[name path = funone, draw=red, mark=none, semithick, domain=-4:-0.15,
         samples=200, -(,shorten >=-2.5pt] 
    {x^2+abs(x)-5};
     \addplot[name path = funtwo, draw=red, mark=none, semithick, domain=0.15:4,
         samples=200,)-,shorten <=-2.5pt] 
    {x^2+abs(x)-5}; 
 \draw[name path = sx, dotted, black] (axis cs:0,-15) -- (axis cs:0,15);
 \draw[name path = sy, dashed, black] (axis cs:-10,-12) -- (axis cs:10,-12);
  \fill[red,name intersections={of=sx and sy}] (intersection-1) circle (2.7pt);
 %\node [above] at (axis cs:-2.2,9.5) {$|z_{v,i}|<w_{v,i}$};
  \node [above] at (axis cs:-1.5,-12) {$z_{v,i}^2$};
  \node [above,red] at (axis cs:-2.5,9.5) {$\sun (x_{v,i})$};
 \end{axis} 
\end{tikzpicture}}
\subfloat[$|z_{v,i}|>w_{v,i}$, $(z_{v,i}-w_{v,i})^2>\mw $]{
\begin{tikzpicture}[scale=0.8]
\begin{axis}[xmin=-3,xmax=5,ymin=-14, ymax=13, xtick={0, 1.5}, ytick={-5, -10}, xticklabels={0, $z_{v,i}-w_{v,i}$},yticklabels={,,}]    
    \addplot[name path = funone, draw=red, mark=none, semithick, domain=-4:-0.035,
         samples=200,-(,shorten >=-2pt] 
    {(2*x-4)^2+4*abs(x)-17};
     \addplot[name path = funtwo, draw=red, mark=none, semithick, domain=0.035:4,
         samples=200,)-,shorten <=-2pt] 
    {(2*x-4)^2+4*abs(x)-17}; 
 \draw[name path = sx, dotted, black] (axis cs:0,-15) -- (axis cs:0,15);
 \draw[name path = sy, dashed, brown] (axis cs:-10,-5) -- (axis cs:10,-5);
 \draw[name path = sminx, dotted, black] (axis cs:1.5,-15) -- (axis cs:1.5,15);
 \draw[name path = sminy, dashed, blue] (axis cs:-10,-10) -- (axis cs:10,-10);
 \fill[red,name intersections={of=sx and sy}] (intersection-1) circle (2.7pt);
 \fill[red,name intersections={of=sminx and sminy}] (intersection-1) circle (2.7pt); 
  \node [below, blue] at (axis cs:-0.9,-10) {$ 2z_{v,i}w_{v,i}-w_{v,i}^2+\mw $};
  \node [below,brown] at (axis cs:-2,-5) {$z_{v,i}^2$};
\end{axis} 
\end{tikzpicture}}
\subfloat[$|z_{v,i}|>w_{v,i}$, $(z_{v,i}-w_{v,i})^2<\mw$]{
\begin{tikzpicture}[scale=0.8]
\begin{axis}[xmin=-3,xmax=5,ymin=-14, ymax=13, xtick={0, 2}, ytick={-6, -10}, xticklabels={0, $z_{v,i}-w_{v,i}$},yticklabels={,,}]    
    \addplot[name path = funone, draw=red, mark=none, semithick, domain=-4:-0.1,
         samples=200,-(,shorten >=-3pt] 
    {(x-4)^2+4*abs(x)-18};
     \addplot[name path = funtwo, draw=red, mark=none, semithick, domain=0.1:5,
         samples=200,)-,shorten <=-3pt] 
    {(x-4)^2+4*abs(x)-18}; 
 \draw[name path = sx, dotted, black] (axis cs:0,-15) -- (axis cs:0,15);
 \draw[name path = sminy, dashed, blue] (axis cs:-10,-6) -- (axis cs:10,-6);
 \draw[name path = sminx, dotted, black] (axis cs:2,-15) -- (axis cs:2,15);
 \draw[name path = sy, dashed, brown] (axis cs:-10,-10) -- (axis cs:10,-10);
  \fill[red,name intersections={of=sx and sy}] (intersection-1) circle (2.7pt);
 \fill[red,name intersections={of=sminx and sminy}] (intersection-1) circle (2.7pt);
 %\node [above] at (axis cs:-2.7,9.5) {$z_{v,i}>w_{v,i}$};
 \node [above, red] at (axis cs:-2,9.5) {$\sun(x_{v,i})$};
 \node [below, blue] at (axis cs:-0.9,-6) {$2z_{v,i}w_{v,i}-w_{v,i}^2+\mw$};
 \node [below, brown] at (axis cs:-2,-10) {$z_{v,i}^2$};
 \end{axis} 
\end{tikzpicture}}
\caption{$\sun(x_{v,i})$  \eqref{fun_dep_xvi} in the cases $|z_{v,i}|<w_{v,i}$ (a) and $|z_{v,i}|>w_{v,i}$ (b)-(c).}\label{fig:1}
\end{figure*}
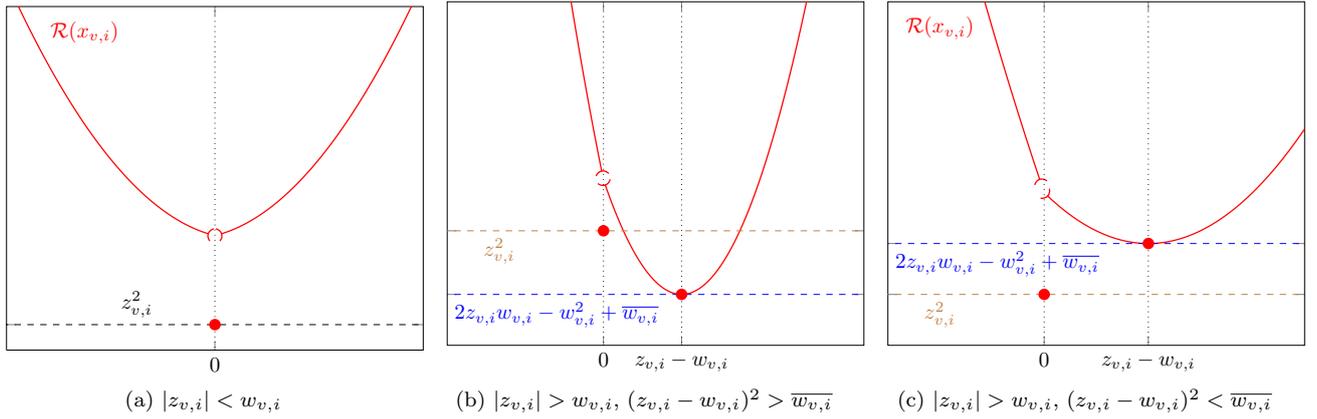

Hence, our procedure to get the minimum of $\sun(x_{v,i})$ is given by the mixed soft/hard thresholding. operator $\soft_{w,a}:\R\mapsto\R$, defined as follows:
\begin{equation}\label{soft_mod}
\soft_{w,a}(x):=\begin{cases}
&0\text{ if } |x|\leq w \text{ or } (x-w)^2\leq a\\
&x-\text{sgn}(x)w\text{ otherwise. } \\
\end{cases}
\end{equation}
This is a slight modification of the well-known soft thresholding operator  $\soft_{w}:\R\mapsto\R$
\begin{equation}\label{soft}
\soft_{w}(x):=\begin{cases}
&0\text{ if } |x|\leq w \\
&x-\text{sgn}(x)w\text{ otherwise. } \\
\end{cases}
\end{equation}
Accordingly, we can write 
 $$x^{+}_{v,i}=\argmin{x_{v,i}\in\R} \sun(x_{v,i})=\soft_{w_{v,i},\mw }(z_{v,i})$$
which, if $\frac{1}{\tau}>\left\|A_v\right\|_2^2$, implies that (\cite[Section 4.1]{for10} for details)    
\begin{align}\label{min_wrt_b}
X=\argmin{B\in\R^{n\times \cardV}}\sun(X,B).
\end{align}
Finally, we conclude that  $\fun$ decreases:
\begin{align}\label{fun_decreases}
\fun(X)=&\sun(X,X)\geq \sun(X^+,X)\\& \geq \sun(X^+,X^+)=\fun(X^+) 	
\end{align}
where $X^+=(x^+_1,\dots,x^+_V)$. The inequality $\sun(X,X)\geq \sun(X^+,X)$ is guaranteed by LLA \cite{can08rew, faz03}. This will be used in next section to prove the convergence.

The procedure outlined above can be summarized as follows: at each iteration step $t$, each node $v$ computes  $x_{v,i}(t+1)=\soft_{w_{v,i}(t),\mwt}(z_{v,i}(t))$, for each $i=1,\dots,n$, where 
$z_v(t)=x_v(t)+\tau A_v^{\mathsf{T}}\big(y_v-A_vx_v(t)\big)$; after that, if $\mathds{1}(x_{v,i}(t+1))\neq \mathds{1}(x_{v,i}(t))$, then $v$ transmits $i$ to its neighbors.

In conclusion, this procedure  solves points 1), 2) and 3) in Section \ref{sec:optimization} by using iterative thresholding.  However, we observed that the soft/hard shrinkage operator $\softmod_{w,a}$ \eqref{soft_mod} tends to oversupply sparsity, which affects the recovery accuracy. To overcome this drawback, we propose to use \eqref{soft} instead of \eqref{soft_mod}, that is, classical soft thresholding. As this may increase $\fun_w(X)$ (specifically, $\fun_w(X(t+1))>\fun_w(X(t))$ when $x_{v,i}(t)=0$,  see Figure \ref{fig:2}.(c)), we allow the switch from zero
to non-zero only for a finite number of times, thus keeping the overall decreasing behavior. In other words, for a finite transient, we  perform soft thresholding; after this transient, the zero components are forced to remain zero. In summary, we update $x_{v,i}(t)$ as follows (see Figure \ref{fig:2}):
\begin{itemize}
 \item if $x_{v,i}(t)\neq 0$, we apply soft thresholding: $x_{v,i}(t+1)=\sigma_{w_{v,i}}(z_{v,t})(t)$. This does not guarantee to get the global minimum of $\sun(x_{v,i})$, but the global minimum or the second minimum: in both cases, we \emph{always} decrease $\sun$;
 \item if $x_{v,i}(t) = 0$, $x_{v,i}(t+1)=\sigma_{w_{v,i}}(z_{v,t})(t)$ for a finite number of times (during which $\sun$ might increase); afterwards, $x_{v,i}(t+1) = 0$.
\end{itemize}
% We indicate the last operation with memory by $\weirdsoft_{w,p}:\R^3\mapsto\R$, defined as
% \begin{equation}
% \weirdsoft_{w,p}(z,x,c):=\begin{cases}
% &0\text{ if } |z|\leq w \text{ or if } x=0 \text{ and } c \geq p\\
% &z-\text{sgn}(z)w\text{ otherwise}. \\
% \end{cases}
% \end{equation}

We  remark again that this transient suboptimal modification i) avoids the  transmission of real values, ii) improves the performance (see Section \ref{sec:numerical_results}), and iii) does not affect the convergence properties of the algorithm (see Section \ref{sec:asymp_reg}).

%%%%%%%%%% FIGURE 2
\begin{figure*}
\center\hspace{-2cm}
	\subfloat[$|z_{v,i}(t)|<w_{v,i}(t)$]{
		\begin{tikzpicture}[scale=0.8]
		\begin{axis}[xmin=-4,xmax=4,ymin=-14, ymax=13, xtick={0,2}, ytick={-12},xticklabels={ {\small{$x_{v,i}(t\hspace{-0.1cm}+\hspace{-0.1cm}1\hspace{-0.03cm})$}}, {\small{$x_{v,i}(t)$}}},yticklabels={,,}] 
		\addplot[name path = funone, draw=red, mark=none, semithick, domain=-4:-0.15,
		samples=200, -(,shorten >=-2.5pt] 
		{x^2+abs(x)-5};
		\addplot[name path = funtwo, draw=red, mark=none, semithick, domain=0.15:4,
		samples=200,)-,shorten <=-2.5pt] 
		{x^2+abs(x)-5}; 
		\draw[name path = sx, dotted, black] (axis cs:0,-15) -- (axis cs:0,15);
		\draw[name path = sy, dashed, black] (axis cs:-10,-12) -- (axis cs:10,-12);
		\fill[red,name intersections={of=sx and sy}] (intersection-1) circle (2.7pt);
	     \draw[blue, thick, -triangle 45,fill=blue] (axis cs:2,-13.5) -- (axis cs:0,-13.5);
         \draw[blue,  thick, dashed, -triangle 45,fill=blue,shorten >=2.5pt] (axis cs:2,1) -- (axis cs:0,-12);
         \draw[black, dotted] (axis cs:2,-13.5) -- (axis cs:2,1);
		\node [above,red] at (axis cs:-2.5,9.5) {$\sun (x_{v,i})$};
		\end{axis} 
		\end{tikzpicture}}\hspace{1cm}
	\subfloat[$|z_{v,i}(t)|>w_{v,i}(t)$, $(z_{v,i}(t)-w_{v,i}(t))^2>\mwt $]{
	 \makebox[3.5cm]{
		\begin{tikzpicture}[scale=0.8]
		\begin{axis}[xmin=-3,xmax=5,ymin=-14, ymax=13, xtick={0, 1.5,3}, ytick={-5, -10}, xticklabels={{\footnotesize{$\hspace{-1cm}x_{v,i}(t)\hspace{-0.1cm}=\hspace{-0.1cm}0$}},{\footnotesize{$x_{v,i}(t\hspace{-0.1cm}+\hspace{-0.1cm}1\hspace{-0.03cm})$}}, {\footnotesize{$\hspace{0.8cm}x_{v,i}(t)\hspace{-0.1cm}\neq\hspace{-0.1cm} 0$}}},yticklabels={,,}]    
		\addplot[name path = funone, draw=red, mark=none, semithick, domain=-4:-0.035,
		samples=200,-(,shorten >=-2pt] 
		{(2*x-4)^2+4*abs(x)-17};
		\addplot[name path = funtwo, draw=red, mark=none, semithick, domain=0.035:4,
		samples=200,)-,shorten <=-2pt] 
		{(2*x-4)^2+4*abs(x)-17}; 
		\draw[name path = sx, dotted, black] (axis cs:0,-15) -- (axis cs:0,15);
		\draw[name path = sy, dashed, black] (axis cs:-10,-5) -- (axis cs:10,-5);
		\draw[name path = sminx, dotted, black] (axis cs:1.5,-15) -- (axis cs:1.5,15);
		\draw[name path = sminy, dashed, black] (axis cs:-10,-10) -- (axis cs:10,-10);
		\fill[red,name intersections={of=sx and sy}] (intersection-1) circle (2.7pt);
		\fill[red,name intersections={of=sminx and sminy}] (intersection-1) circle (2.7pt); 
		\node [above, red] at (axis cs:-1.8,9.5) {$\sun(x_{v,i})$};
		\draw[blue, thick, -triangle 45,fill=blue] (axis cs:3,-13.5) -- (axis cs:1.5,-13.5);
        \draw[blue,  thick, dashed, -triangle 45,fill=blue,shorten >=2.5pt] (axis cs:3,-1) -- (axis cs:1.5,-10);
        % \draw [-to,shorten >=-1pt,gray,ultra thick] (axis cs:0.1,-14) -- (axis cs:0,-14); 
         \draw[black, dotted] (axis cs:3,-13.5) -- (axis cs:3,-1);
         \draw[blue, thick, -triangle 45,fill=blue] (axis cs:0,-13.5) -- (axis cs:1.5,-13.5);
        \draw[blue,  thick, dashed, -triangle 45,fill=blue,shorten >=2.5pt] (axis cs:0,-5) -- (axis cs:1.5,-10);
        % \draw [-to,shorten >=-1pt,gray,ultra thick] (axis cs:0.1,-14) -- (axis cs:0,-14); 
         \draw[black, dotted] (axis cs:3,-13.5) -- (axis cs:3,-1);
		\end{axis} 
		\end{tikzpicture}}}\hspace{3cm}
	\subfloat[$|z_{v,i}(t)|>w_{v,i}(t)$, $(z_{v,i}(t)-w_{v,i}(t))^2<\mwt $]{\makebox[3.6cm]{
		\begin{tikzpicture}[scale=0.8]
		\begin{axis}[xmin=-3,xmax=5,ymin=-14, ymax=13, xtick={0, 2,4}, ytick={-6, -10}, xticklabels={{\footnotesize{$\hspace{-1cm}x_{v,i}(t)\hspace{-0.1cm}=\hspace{-0.1cm}0$}},{\footnotesize{$x_{v,i}(t\hspace{-0.1cm}+\hspace{-0.1cm}1\hspace{-0.03cm})$}}, {\footnotesize{$\hspace{0.1cm}x_{v,i}(t)\hspace{-0.1cm}\neq\hspace{-0.1cm} 0$}}},yticklabels={,,}]    
		\addplot[name path = funone, draw=red, mark=none, semithick, domain=-4:-0.1,
		samples=200,-(,shorten >=-3pt] 
		{(x-4)^2+4*abs(x)-18};
		\addplot[name path = funtwo, draw=red, mark=none, semithick, domain=0.1:5,
		samples=200,)-,shorten <=-3pt] 
		{(x-4)^2+4*abs(x)-18}; 
		\draw[name path = sx, dotted, black] (axis cs:0,-15) -- (axis cs:0,15);
		\draw[name path = sminy, dashed, black] (axis cs:-10,-6) -- (axis cs:10,-6);
		\draw[name path = sminx, dotted, black] (axis cs:2,-15) -- (axis cs:2,15);
		\draw[name path = sy, dashed, black] (axis cs:-10,-10) -- (axis cs:10,-10);
		\fill[red,name intersections={of=sx and sy}] (intersection-1) circle (2.7pt);
		\fill[red,name intersections={of=sminx and sminy}] (intersection-1) circle (2.7pt);
		\node [above, red] at (axis cs:-2,9.5) {$\sun(x_{v,i})$};
		;
	%%%%%%%
		  \draw[orange, thick, dashed, -triangle 45,fill=orange,shorten >=2.5pt] (axis cs:.0,-10) -- (axis cs:2,-6);
        \draw[blue,  thick, -triangle 45,fill=blue] (axis cs:0,-13.5) -- (axis cs:2,-13.5);
        % \draw [-to,shorten >=-1pt,gray,ultra thick] (axis cs:0.1,-14) -- (axis cs:0,-14); 
        % \draw[black, dotted] (axis cs:3,-13.5) -- (axis cs:3,-1);
      \draw[blue, thick, dashed, -triangle 45,fill=blue,shorten >=2.5pt] (axis cs:4,-2) -- (axis cs:2,-6);
        \draw[blue,  thick, -triangle 45,fill=blue] (axis cs:4,-13.5) -- (axis cs:2,-13.5);
         \draw[black, dotted] (axis cs:4,-2) -- (axis cs:4,-13.5);
		\end{axis} 
		\end{tikzpicture}}}
	%\caption{Case 3}}
	\caption{ Dynamics of $\sun(x_{v,i}(t))$ when $x_{v,i}(t+1)=\soft_{w_{v,i}}(x_{v,i}(t))$. The arrows depict the movements of $x_{v,i}(t)$ and $\sun(x_{v,i}(t))$. In the case (c), if $x_{v,i}(t)=0$, $\sun(x_{v,i}(t))<\sun(x_{v,i}(t+1))$ (orange arrow). This increasing movement is allowed only for a finite number of times, after which if $x_{v,i}(t)=0$, we fix   $x_{v,i}(t+1)=0$. In this way, the definitive behavior of  $\sun(x_{v,i}(t))$ is non-increasing.}
	\label{fig:2}
\end{figure*}
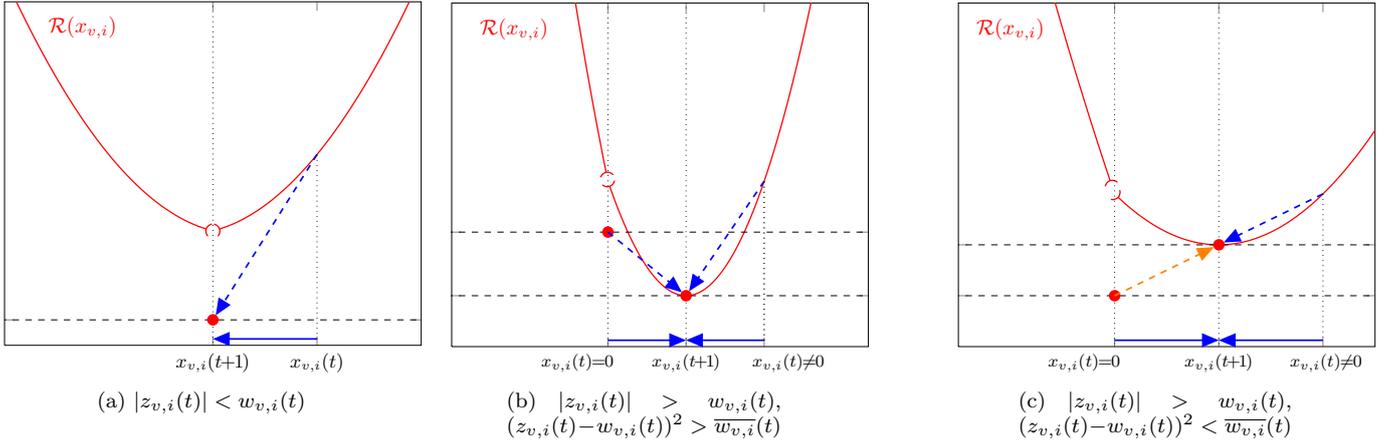
Bearing all the above in mind, our distributed procedure for the recovery of jointly sparse signals based on IST, DJ-IST in short, is described in Algorithm \ref{DJ-IST}. %DJ-IST actually turns out to be Algorithm \ref{DJ-IST_0} with the transient modification described above (see steps 8 to 11).

\begin{algorithm}
\caption{DJ-IST}\label{DJ-IST}
\begin{algorithmic}[1] 
\STATE   Initialize   variables:  \\For all $v\in\V$,  $x_v(0)=A_v^\mathsf{T}y_v$;
$s_v(0)=[1,1,\dots,1]^{\mathsf{T}}$; $p\in\N$ (finite); $\epsilon>0$, $\tau>0$, $\lambda>0$, $\alpha>0$
\STATE $t=0$
\FORALL{$v\in\V$}
\STATE  $z_v(t)=x_v(t)+\tau A_v^{\mathsf{T}}\big(y_v-A_vx_v(t)\big)$
\FORALL{$i=1,\dots,n$}
%
%\STATE  Update support estimate:\\ $s_{v,i}(t)=\min\{1,\overline{\mathds{1}(x_{v,i}(t))}+2\alpha|x_{v,i}(t)|\}$
%
%\STATE  $w_{v,i}(t)=\lambda+\alpha(1-s_{v,i}(t)) $
%
\STATE Update threshold  $w_{v,i}(t)=[\beta -\alpha | x_{v,i}(t)| -\overline{\mathds{1}(x_{v,i}(t)}]_+ $
%
%\STATE Update signal estimate:\\ $x_{v,i}(t+1)=\weirdsoft_{w_{v,i}(t),p}(z_{v,i}(t),x_{v,i}(t),c_{v,i}(t))$
%
\STATE Update signal estimate:\\ $x_{v,i}(t+1)=\soft_{\lambda\alpha w_{v,i}(t)}(z_{v,i}(t))$

\IF{ $x_{v,i}(t)=0$ and $c_{v,i}(t)\geq p$}
\STATE      $x_{v,i}(t+1)=0$
\ENDIF

\IF{ $x_{v,i}(t)=0$ and $x_{v,i}(t)\neq0$ }
\STATE $c_{v,i}(t+1)=c_{v,i}(t)+1$
\ENDIF
\IF{ $\mathds{1}(x_{v,i}(t+1))\neq \mathds{1}(x_{v,i}(t))$}
\STATE Transmit index $i$ to the neighbors 
\ENDIF
\ENDFOR
\IF{$\left\|x_v(t+1)-x_v(t)\right\|_2<\epsilon$}
\STATE Node $v$ stops
\ELSE
\STATE $t\longleftarrow t+1 $
\ENDIF 

\ENDFOR 

\end{algorithmic}
\end{algorithm} 

It is worth noting that DJ-IST merely requires to transmit information about the support, specifically,  the indices of the components that switched from zero to non-zero and vice versa.  Since the sensor signals $x_v$'s are in $\R^n$, DJ-IST transmits $\lfloor \log_2 n \rfloor +1$ bits for each switched component. 

%Regarding the parameters involved in DJ-IST, we point out that $\lambda$ in \eqref{def:weight} cannot be %set to zero just because we also need  sparsity promotion at the beginning, say when the $s_{v,i}(t)$ are all forced to be  equal to 1 (see step 1 in Algorithm \ref{DJ-IST}). We also remark the importance of the interplay between $\lambda$ and $\alpha$: as we can see in step 7,  they both have a role in tuning the threshold, and by increasing $\alpha$ we give more weight to the role of the  support estimate $s_{v,i}(t)$.

\subsection{Other iterative algorithms for Lasso}\label{sub:why_not_admm}
At the beginning of the section, the use of iterative thresholding was naturally motivated by its adaptability to decrease the non-convex functional $\fun$ in  \eqref{ultimate_problem}, which presents  $\ell_1$ and $\ell_0$ penalization terms.

In the literature,  methods faster than iterative thresholding have been proposed to solve convex problems as Lasso. For example, the alternating direction method of multipliers (ADMM, \cite{boy10,yan11}), and the fast iterative thresholding algorithm (FISTA, \cite{bec09}) have been shown to be very efficient. In principle, such methods   cannot be applied to \eqref{ultimate_problem} due to the non-convexity of $\fun$. Through this section, however, we have reduced the step that updates $X(t)$ to an IST step (with forced stabilization of the null components after a finite transient), which means that we simply decrease the Lasso part of $\fun$, and the role of $\ell_0$ is only to stop the switches from zero to non-zero.

From this perspective, we could consider again methods as ADMM and FISTA to update $X(t)$. However, we observed that such methods are somehow too fast for our problem. In fact, if the procedure is too fast, nodes tend to estimate their signals support based on their local measurements, \emph{i.e.}, without taking into account other nodes information. This  causes some transient instability in which many support switches occur, which implies many more transmissions, thereby penalizing the communication cost. In conclusion more conservative methods ultimately reduce the number of transmissions, which makes the slow IST more efficient. In order to illustrate these observations, in Section \ref{sub:ADMM} we will show some numerical simulations based on ADMM.

\section{Convergence of DJ-IST}
In this section, we prove that DJ-IST converges. We first show the the numerical convergence and the support stabilization, and we then exploit them to prove the point convergence.

\subsection{Numerical convergence }\label{sec:asymp_reg}

We now prove the numerical convergence (or asymptotic regularity) of the sequence $X(t)$ produced step by step by Algorithm \ref{DJ-IST}, namely
\begin{align*}
 \lim_{t\rightarrow+\infty}&\left\|X(t+1)-X(t)\right\|_F^2=0
% \lim_{t\rightarrow+\infty}&\left\|W(t+1)-W(t)\right\|_F^2=0.
\end{align*}
%Notice that this condition substantiates the stopping criterion of Algorithm \ref{DJ-IST} (see step 16).

We remark that for the convergence analysis we do not take into consideration the fact that for a finite number of steps, increases of $\sun$ are allowed, as they clearly have no effect on the asymptotic properties of the algorithm. From now on, we then consider $t\geq t_0$, where $t_0$ is any fixed time step after the finite transient.

\begin{proposition}\label{decrescenza} 
Given the sequence $X(t)$ generated by DJ-IST (Algorithm \ref{DJ-IST}), $\{\fun(X(t))\}_{t\in\N}$ for $t\geq t_0$ is non-increasing, and admits the limit. Moreover, if  $\tau  <\left\|A_v\right\|_2^{-2}$ for  all $v\in\mathcal{V}$, $X(t)$ is numerically convergent.
\end{proposition}
\begin{proof}
By \eqref{fun_decreases} and following discussion, for any for $t\geq t_0$,
$\fun(X(t))\geq \fun(X(t+1)) $, that is, $\fun(X(t))$ is non-increasing. As it is lower bounded ($\fun(X)\geq 0$ for any $X\in\R^{n\times\cardV}$), then it admits the limit.
Hence, $\fun(X(t))-\fun(X(t+1))\rightarrow 0$. On the other hand, 
%{\small{
\begin{align*}
&\fun(X(t))-\fun(X(t+1))\\& = \sun(X(t),  X(t))-\sun(X(t+1), X(t+1))\\
&\geq \sun(X(t+1), X(t))-\sun(X(t+1), X(t+1))\\
&\geq \sum_{v\in\mathcal{V}}(x_v(t+1)-x_v(t))^{\mathsf{T}}(I-\tau  A_v^{\mathsf{T}}A_v)(x_v(t+1)-x_v(t))\\
&\geq 0.
\end{align*}
%}}
The  last inequality is due to the positive definiteness of  $I-\tau  A_v^{\mathsf{T}}A_v$ guaranteed by  the hypothesis $\tau  <\left\|A_v\right\|_2^{-2}$. We thus conclude  that $\left\|x_v(t+1)-x_v(t)\right\|_2^2\to 0$  for any
$v\in\mathcal V$ and that 
$\lim_{t\rightarrow+\infty}\left\|X(t+1)-X(t)\right\|_F^2=0.$

\end{proof}

Furthermore, we can easily observe that support stabilizes at a finite time.
\begin{theorem}\label{theo:supp_stab}
There exists a time $t_1\in\N$ at which the sequence $\mathds{1}(X(t))$ stabilizes, that is,  $\mathds{1}(X(t))$ is constant for any $t\geq t_1$.
\end{theorem}

\begin{proof}
After a finite number of allowed switches, no more switches from zero to non-zero are possible for DJ-IST, say $x_{v,i}(t)=0$, then $x_{v,i}(t+1)=0$ , for any $v\in\V$, $1=1,\dots,n$. This is sufficient to state that the support stabilizes. In particular, we call $t_1$ the time at which all the components of all the nodes have stabilized their status.
\end{proof}

Alternatively, this result could be easily deduced from Proposition \ref{decrescenza}. Since $X(t)$ numerically converge and the support stabilize, we notice that also $W(t)$ numerically converge.
%Once the support estimation has stabilized, our main goal should be considered achieved.  No more communication is necessary and the signal estimate (say, the estimate of the non-zero values) could be performed by each node individually by a least squares method, as done  in \cite{wima14}.

% However, with DJ-IST it  is not necessary to split the recovery into two different procedures, one for the estimate of the support and one for the estimate of the non-zero values. Notice that splitting the recovery into two different procedures is more critical when $k$ is not known, as there is  no secure criterion to understand when the support has stabilized.
% We now show that one can run DJ-IST also after the support stabilization and get the convergence of $X(t)$.  In the previous section, we have already  proved that the numerical convergence occurs, thus  a practical stopping criterion is provided: at any $t\in\N$, each node should store $x_v(t)$ and $x_v(t-1)$ and stop when the distance between the two iterates is below a fixed threshold depending on the machine epsilon.
% 
% In addition, in the next we propose the rigorous proof of the convergence and give a description of the convergence points. 

\subsection{Point convergence}\label{sec:convergence}
We now leverage numerical convergence and support stabilization to prove rigorous point convergence. 
%The first step in this direction is to show that the support estimate $\mathds{1}(x_v(t))$, $v\in\V$, stabilizes. 
%After support stabilization, no more communication is needed and each node goes ahead individually, varying the non-zero components according to a continuous dynamics. 

Once the support estimation has stabilized, our main goal should be considered achieved.  No more communication is necessary and the signal estimate (say, the estimate of the non-zero values) could be performed by each node singularly by a least squares method, as done  in \cite{wima14}.

However, with DJ-IST it  is not necessary to split the recovery into two different procedures, one for the estimate of the support and one for the estimate of the non-zero values. Notice that splitting the recovery into two different procedures is more critical when $k$ is not known, as there is  no secure criterion to establish when the support has stabilized.
We now show that one can run DJ-IST also after the support stabilization and get the convergence of $X(t)$.  In the previous section, we have already  proved the numerical convergence, which provides a practical  stopping criterion: at any $t\in\N$, each node should store $x_v(t)$ and $x_v(t-1)$ and stop when the distance between the two iterates is below a fixed threshold depending on the machine epsilon. In this section, we  propose a rigorous point convergence proof and give a description of the convergence points.

%\subsection{Convergence of $X(t)$}
Let us consider the system evolution after support stabilization.  First of all, we notice that the problem is no more distributed: communications actually stop and each  node $v\in\V$ proceeds individually. 

As the zeros are now fixed, let us now describe  the evolution of the non-zero components of each $v$. Let us call $\widehat{\Omega}_v\subset \{1,\dots, n \}$ the active set, \emph{i.e.} the estimated support for node $v$, which is constant after support stabilization. We define the partition: $\widehat{\Omega}_v=\widehat{\Omega}_{v,1}(t)\cup \widehat{\Omega}_{v,2}(t)$ where
 $$\widehat{\Omega}_{v,1}(t):=\{i\in\{1,\dots,n\}\text{ s.t. } w_{v,i}(t)>0\}$$ and $\widehat{\Omega}_{v,2}:=\widehat{\Omega}_v\setminus \widehat{\Omega}_{v,1}$, that is, $$\widehat{\Omega}_{v,2}(t)=\{i\in\{1,\dots,n\}\text{ s.t. } w_{v,i}(t)=0\}.$$

First, we remark that the signs of the non-zero components are definitely constant. To see this, suppose the sign changes in the next iteration, \emph{e.g.} $x_{v,i}(t)>0$ and $x_{v,i}(t+1)<0$. Given the numerical convergence, large deviations between consecutive iterations are not possible, and thus we expect $x_{v,i}(t)\in\widehat{\Omega}_{v,1}(t)$, so that $x_{v,i}(t)<\frac{\beta-\overline{\mathds{1}(x_{v,i}(t))}}{\alpha}$. We have then $w_{v,i}(t)>0$, and in particular the more $x_{v,i}(t)$ is close to zero, the more $w_{v,i}(t)$ is large, then we can consider $w_{v,i}(t)\geq \epsilon>0$. To switch the sign we must have $z_{v,i}(t)>\alpha w_{v,i}(t)$ and $z_{v,i}(t+1)<-\alpha w_{v,i}(t)$; however, this is not possible as $z_{v,i}(t)$ numerically converges as well, and after a finite time it cannot overstep an interval of length $2\alpha w_{v,i}(t)>2\epsilon>0$. Following this rationale, an intermediate step  in which $|z_{v,i}(t)|<\alpha w_{v,i}(t)$ is expected, which entangles $x_{v,i}(t)$ into zero.

Bearing this in mind, the evolution of the non-zero components can be expressed as follows.  Let $ A_{\widehat{\Omega}_v}$ be $A_v$ limited to the columns that belong to $\widehat{\Omega}_v$. We have 
\begin{equation}\label{map_gamma}
\begin{split}
 \Gamma_v:&  \R^{\widehat{k}_v}\mapsto \R^{\widehat{k}_v}\\
 \Gamma_v(x)&=M_v(x) x+c_v(x)
 \end{split}
 \end{equation}
 where 
 \begin{equation*}
\begin{split}
&M_v(x)\in \R^{\widehat{k}_v\times \widehat{k}_v},~~~M_v(x)=\alpha^2 D_{v}(x) +I_v-\tau  A_{\widehat{\Omega}_v}^{\mathsf{T}} A_{\widehat{\Omega}_v}  \\
&c_v(x)\in \R^{\widehat{k}_v},~~~c_v(x)=- D_{v}(x)\alpha\text{sgn}(x)(\beta-\widehat{\mathds{1}}_v)+\tau  A_{\widehat{\Omega}_v}^{\mathsf{T}} y_v 
\end{split}
\end{equation*}
and  $I_v$ is the identity matrix of dimensions  $\widehat{k}_v\times \widehat{k}_v$;  $ D_{v}(x)$ is the binary diagonal matrix which has a 1 in position $(i,i)$ if $x_{v,i}\in \widehat{\Omega}_{v,1}$, and zero otherwise;  $\widehat{\mathds{1}}_{v,i}=\overline{\mathds{1}(x_{v,i}(t_1))}$, where $t_1$ is the support stabilization time, then    $\widehat{\mathds{1}}_v$  is  constant.

$M_v(x)$ is positive definite for any $x\in\R^{\widehat{k}_v}$, and whenever a component of $x_v$ is in $\widehat{\Omega}_{v,1}$, the transition matrix $M_v(x)$ is expansive if $  A_{\widehat{\Omega}_v}^{\mathsf{T}} A_{\widehat{\Omega}_v}$ has not  maximum rank. Iterating $\Gamma_v(x)=M_v(x) x+c_v(x)$ we then expect that all the components of $x_v$ will blow up at infinity, but actually this is not the case because when $|x_{v,i}|>\frac{\beta-\widehat{\mathds{1}}_{v,i}}{\alpha}$, we move to regime $\widehat{\Omega}_{v,2}$, in which the system turns out to be a simple gradient descent that converges to a minimum of $\|  A_{\widehat{\Omega}_v} x-y_v\|$. This proves the following Lemma. 
\begin{lemma}
For any $v\in\V$, $t\in\N$, $x_{v}(t)$ is bounded.
\end{lemma}

The dynamical system of \eqref{map_gamma} is a switched linear system: when $x_{v,i}(t)$ switches from $\widehat{\Omega}_{v,1}$ to $\widehat{\Omega}_{v,2}$, the entry $(i,i)$ of $D_v$ switches from 1 to 0, and vice versa. Possible oscillations between the two regions make  the convergence proof more complicated and technical. To simplify it, we do the following realistic assumption. 

\begin{assumption}\label{cosi_contraggo}
For any $v\in\V$ and $t\in\N$,    $\max |x_{v,i}(t)|<\frac{\beta-\widehat{\mathds{1}}_{v,i}}{\alpha}$, that is, $x_{v,i}(t)\in \widehat{\Omega}_{v,1}$.
\end{assumption}

This assumption is commonly fulfilled as generally we  set $\alpha$ much smaller than $\beta$.  Therefore, $|x_{v,i}(t)|\geq \frac{\beta-\widehat{\mathds{1}}_{v,i}}{\alpha}$ implies $\fun(X)(t)$ of the order of $\frac{\beta-\widehat{\mathds{1}}_{v,i}}{\alpha}$, which is very high. For example, in our simulations (Section \ref{sec:numerical_results}), we set $\alpha=5\cdot 10^{-4}$ and $\beta=1.1$, which implies $\fun(X(t))$ of order $10^6$ for $|x_{v,i}(t)|\geq \frac{\beta-\widehat{\mathds{1}}_{v,i}}{\alpha}$. Therefore, it suffices to set a reasonable initial condition to have $\fun(X(0))$ smaller than such values: since $\fun(X(t))$ is not increasing, this guarantees that $|x_{v,i}(t)$ will never exceed $\frac{\beta-\widehat{\mathds{1}}_{v,i}}{\alpha}$.
% 
% In the practice, this assumption can be easily fulfilled by choosing a suitable initial value $x_{v,i}(0)$. Let us assume that $A_vx_v$ and $\beta>1$ are of the order of the unity and that $\alpha$ is much smaller (so that $A_vx_v$ and $\beta$ are negligible with respect to $\frac{1}{\alpha}$). If at least a  single $x_{v,i}\in  \widehat{\Omega}_{v,2}$, the cost functional $\fun$ is at least of the order of  $f_1=\tau \frac{(\beta-1)^2}{\alpha^2}$. Let us instead fix  $x_{v}=A_v^{\mathsf{T}}y$, and let us assume all components different from zero (for example, if $A_v$ is Gaussian, this is true with probabilty 1), $\fun$ will be of the order $f_2=n\cardV\left[\tau+ \beta-\frac{1}{2}\right]$. It is easy to show that, for a sufficiently small $\alpha$,  $f_2<f_1$, and the decreasing behavior of the algorithm will never lead to $f_1$, say no component can escape to $\widehat{\Omega}_{v,2}$. Notice that this observation entails also the allowed non optimal steps, that may switch the indicator functions from 0 to 1 (while the rest of the functional always decreases), as our starting point has all the indicator functions equal to 1. In the practice the requiremnt is approximately, $\alpha^2<\frac{\tau (\beta-1)}{n\cardV}$. This will be respected in our simulations (Section \ref{sec:numerical_results}), where we will fix $\alpha\approx 10^{-4}$, $\beta-1 \approx 10^{-1}, n=100, \cardV=10, \tau\approx 10^{-3}$.

Under Assumption \ref{cosi_contraggo}, the evolution of our system is simply linear: 
\begin{equation}\label{map_gamma_contrattiva}
\begin{split}
 \Gamma_v:&  \R^{\widehat{k}_v}\mapsto \R^{\widehat{k}_v}\\
 \Gamma_v(x)&=M_v x+c_v
 \end{split}
 \end{equation}
 where 
 \begin{equation*}
\begin{split}
M_v&= (\alpha^2+1)I_v-\tau  A_{\widehat{\Omega}_v}^{\mathsf{T}} A_{\widehat{\Omega}_v} \in \R^{\widehat{k}_v\times \widehat{k}_v} \\
c_v&=\alpha\text{sgn}(x)(\beta-\widehat{\mathds{1}}_v)+\tau  A_{\widehat{\Omega}_v}^{\mathsf{T}} y_v \in \R^{\widehat{k}_v}
\end{split}
\end{equation*}

From previous observations, we know that $x_{v,i}(t)$ is bounded, so such $M_v$ cannot be expansive. We therefore conclude that $ A_{\widehat{\Omega}_v}^{\mathsf{T}} A_{\widehat{\Omega}_v}$ must have maximum rank. Assuming that the components of $A_v$ are randomly chosen according a continuous distribution,  $A_v^{\mathsf{T}}A_v$ has rank $m$; since  $ A_{\widehat{\Omega}_v}^{\mathsf{T}} A_{\widehat{\Omega}_v}$ has  dimension  $\widehat{k}_v$, we conclude that it can have maximum rank $\widehat{k}_v$ only if  $\widehat{k}_v\leq m$. We observe that this makes sense, as this is the case for  iterative soft thresholding \cite{tib13}, which is the basis for our algorithm. This condition is necessary but also sufficient to have maximum rank, provided that $\tau A_{\widehat{\Omega}_v}^{\mathsf{T}} A_{\widehat{\Omega}_v}$ has no eigenvalues equal to $\alpha^2$ (if $A_v$ is random, this occurs with probability 0). Moreover, if $\alpha$ is sufficiently small, we have $\left\|M_v\right\|_2<1$

Finally, we have the following convergence theorem.
 \begin{theorem}\label{theo:conv_fp}
 For a sufficiently small $\alpha$, the sequence $X(t)$ generated by DJ-IST (Algorithm \ref{DJ-IST}) converges to a local minimum of $\fun(X)$.
 Moreover, for each $v\in\V$, the non-zero components of $x_v(t)$ converge to
\begin{equation}\label{puntofisso}
 [I_v-M_v]^{-1}c_v=[\tau  A_{\widehat{\Omega}_v}^{\mathsf{T}} A_{\widehat{\Omega}_v}-\alpha^2 I_v]^{-1}[\alpha s_v(\beta-\widehat{\mathds{1}}_v)+\tau  A_{\widehat{\Omega}_v}^{\mathsf{T}} y_v ]\end{equation}
where $s_v=\text{sgn}(x_v(t_1))$, $t_1$ being the support stabilization time.
\end{theorem}
\begin{proof}
For a sufficiently small $\alpha$, $\left\|M_v\right\|_2<1$, that is, the map \eqref{map_gamma_contrattiva} is contractive. Therefore, a fixed point exists and convergence to it is guaranteed (no matter which is the initial point) by the Banach fixed-point theorem. In particular, iterating the map \eqref{map_gamma_contrattiva} we obtain a geometric series that converges to \eqref{puntofisso}. 

This concludes the convergence of the non-zero components, which along with support stabilization proved in Theorem \ref{theo:supp_stab} gives the convergence.

We  remark that the point \eqref{puntofisso} turns out to be the unique minimum of 
$$\tau\left\| A_{\widehat{\Omega}_v} x-y_v\right\|+\sum_{i=1}^n 2\beta \alpha|x_i|-\frac{1}{2}\alpha^2x_i^2$$
and, as a consequence, a local minimum of $\fun (X)$. In fact, if we  perturb the non null components we increase $\fun$ due to the last statement, while if we perturb the zero components, the indicator function switch to 1 and cause a sure increase of $\fun$.
% We notice that if $\tau$ and $\alpha$ are sufficiently small, $\left\|M_v\right\|_2\leq 1$. Thus, for any $v\in\V$, $\Gamma_v$ is then non-expansive.
% Since any affine non-expansive map admits a fixed point \cite[Theorem 2.4]{dur01}, $\Gamma_v$ is in the hypotheses of the Opial's theorem \cite[Theorem 4.6]{for10} (the asymptotic regularity was discussed in Section \ref{sec:asymp_reg}), which guarantees convergence to a fixed point.
\end{proof}

Regarding the convergence point \eqref{puntofisso}, we observe that this coincides with the true value if $x^{\star}_v=\frac{1}{\alpha} s_v(\beta-\widehat{\mathds{1}}_v)$, otherwise a bias is present. This was expected as $\ell_1$ minimum is known to be bias proportionally to the the $\ell_1$ weight. In our reweighted $\ell_1$ setting, however an accurate choice of $\beta$ and $\alpha$ could reduce this bias. Such optimization will be focus of our future work.

\section{Numerical results}\label{sec:numerical_results}
In this section, we show the results of some numerical simulations and compare the performance of DJ-IST with the state-of-the-art algorithms DC-OMP 1 and DC-OMP 2 \cite{wima14}. 
\subsection{DC-OMP 1 and DC-OMP 2}
The rationale behind DC-OMP 1 \cite[Algorithm 3]{wima14} is the following: each node performs a step of OMP and computes an index candidate (by evaluating the largest correlation between residual and columns of the sensing matrix) to add to the support; the candidates are then locally shared, and the candidates with more than one occurrence are added to the support, except for the case that those candidates do not change the support (in this case, each node introduces its own candidate); if all the candidates have one occurrence, each node adds its own candidate. A slight modification is considered when the communication is complete. Notice that DC-OMP 1 is very similar to DiOMP \cite{sun14}, with some differences in the voting procedure, which makes DC-OMP 1 more reliable. 
In DC-OMP 2 \cite[Algorithm 4]{wima14}, instead,  each node locally shares not only the index candidate, but all the correlations between  residual and columns of its sensing matrix. The index candidate is then chosen fusing the correlations  and then transmitted to all the network via multi-hop communication. In DC-OMP 2 more information is shared with respect to DC-OMP 1, then better performance can be expected.

The goal of this section is to numerically prove that DJ-IST is a good trade-off between DC-OMP 1 and DC - OMP 2, in terms of support reconstruction accuracy and use of the communication links.

\subsection{Simulations setting}\label{sub:simulations_setting}
For all our experiments, the original signals $x^{\star}_v$ have joint support   generated uniformly at random, and the non-zero elements are drawn from a standard Gaussian
distribution. The entries of the sensing matrices are generated according to a standard Gaussian
distribution as well, and then normalized by $\sqrt{m}$. Results are averaged over 250 different runs, obtained by generating 50 different sets of  $x^{\star}_v$ and trying $5$ different sensing matrices for each. We stop the algorithm at time $T=\min\{t\in\N\text{ s.t. } |x_{v,i}(t+1)-x_{v,i}(t))|<\epsilon=10^{-5}$,  for all  $v\in\V, i=1,\dots,n\}$. The parameters $\lambda$, $\alpha$, $\beta$ and $\tau$ have been empirically set; in all our simulations, 
$\lambda=1$, $\alpha=5\times 10^{-4}$, $\beta=1.1$, $\tau=2e-2$. The parameter $p$ is not actually fixed, as naturally few switches from zero to non-zero occur (in all our simulations, we observed at most 9 switches).

\subsection{Support recovery performance}

We evaluate two performance metrics for the support: the average support error (ASE), defined as

\begin{equation}
\text{ASE}=  \sum_{v\in\V}\frac{\left\|\mathds{1}(x^{\star}_v) - \widehat{\omega}_v  \right\|_0}{n\cardV}
\end{equation}

and the probability of exact support recovery (PESR)
\begin{equation}
\text{PESR}=  \sum_{v\in\V}\frac{\mathds{I}( \mathds{1}(x^{\star}_v) - \widehat{\omega}_v)  }{\cardV}
\end{equation}
where $\mathds{I}(x)$ is the function from $\R^n$ to $\R$ that returns 1 when the vector $x=(0,0,\dots,0)^{\mathsf{T}}\in\R^n$  and 0 otherwise. PESR assesses how many sensors  estimate the right support, while ASE measures how large is the error in the support for each sensor, on average. 
\begin{figure*}[t]
\centering
\includegraphics[width=0.95\columnwidth]{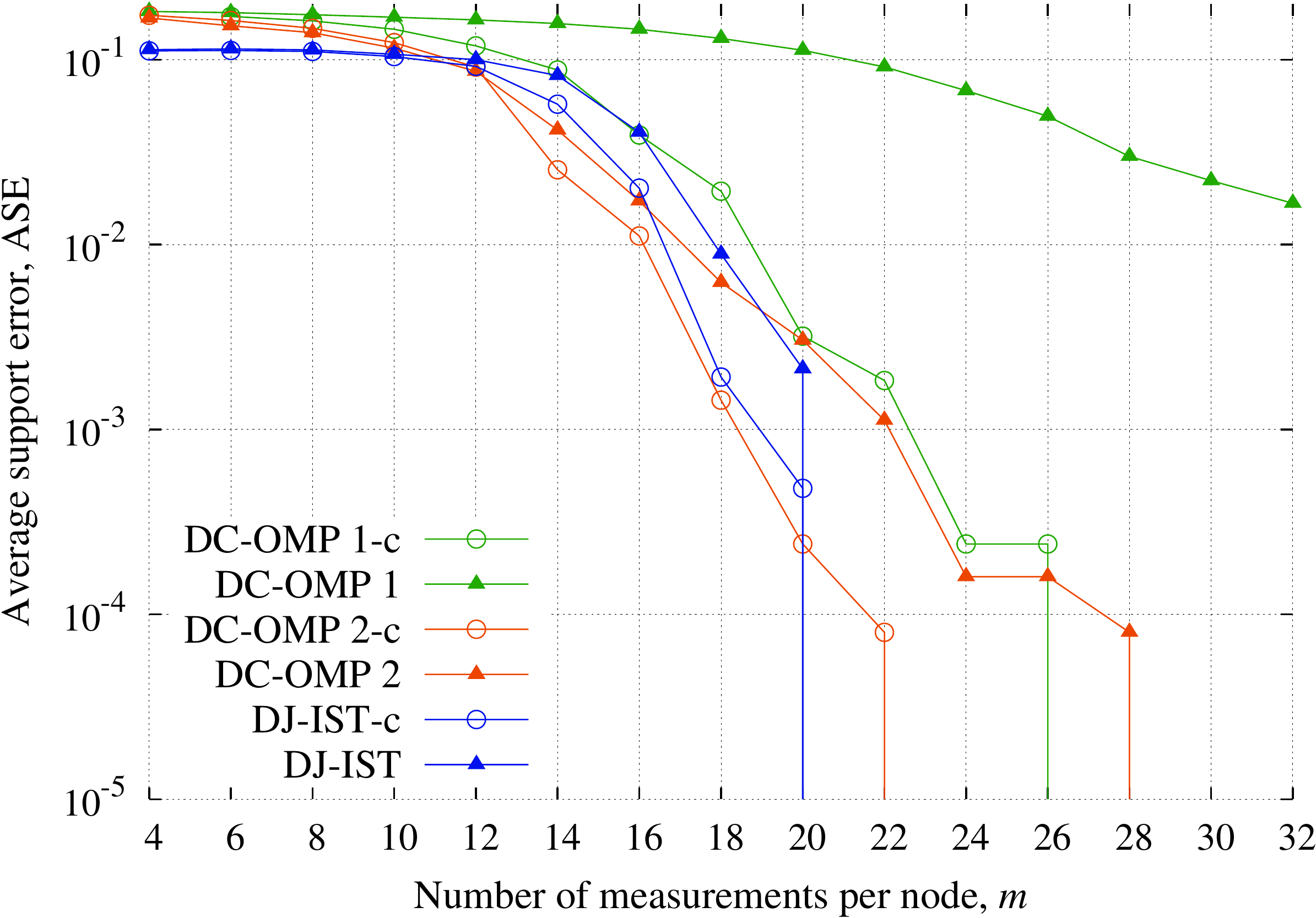}\qquad
\includegraphics[width=0.95\columnwidth]{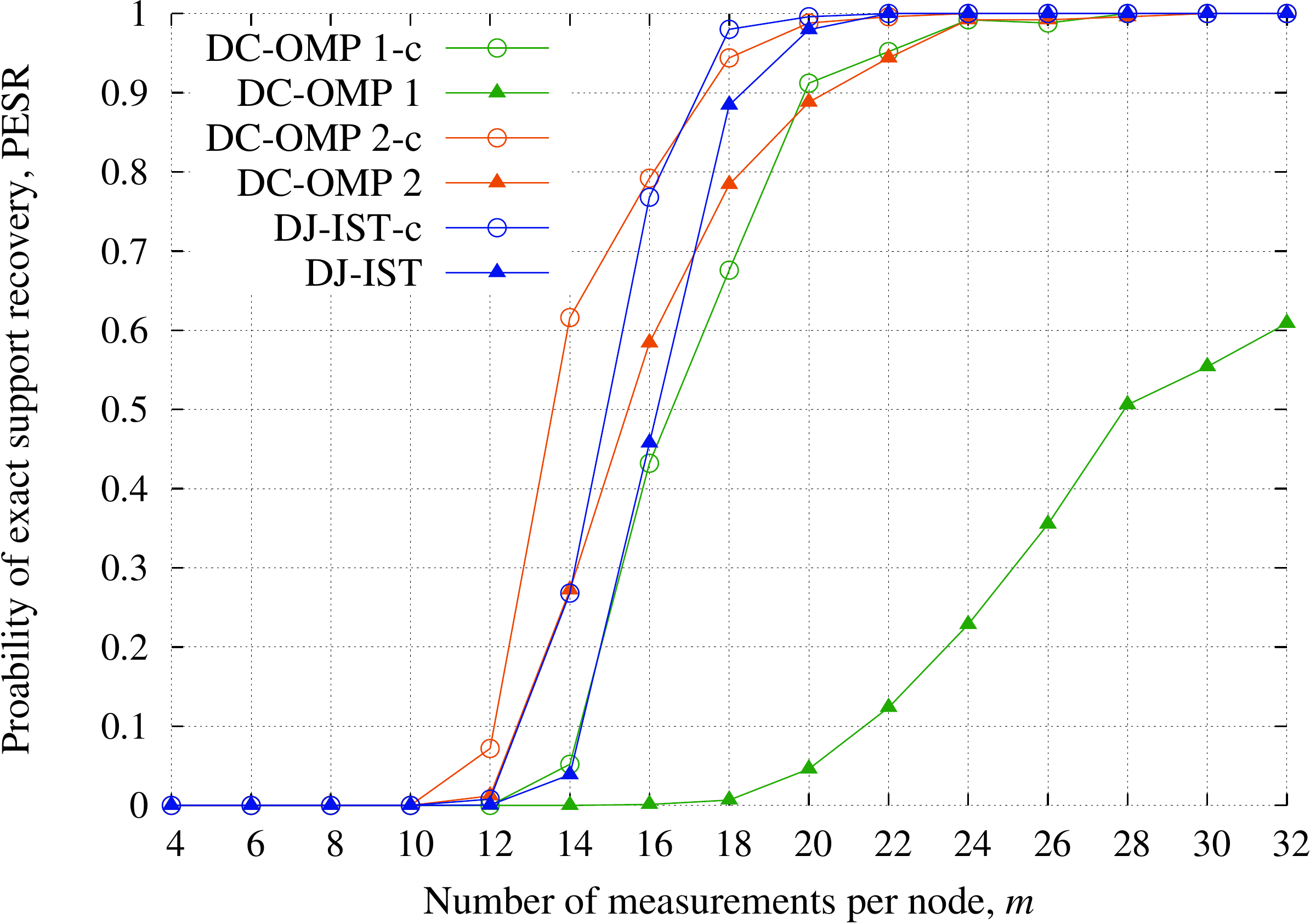} 
\caption{ASE (left) and PESR (right) as a function of $m$, $V=10$,  $\lambda=1$, $\alpha=5\times 10^{-4}$, $\beta=1.1$, $\tau=2\times 10^{-2}$.}
\label{fig:measNode}
\end{figure*}
\begin{figure*}[ht]
\centering
\includegraphics[width=0.95\columnwidth]{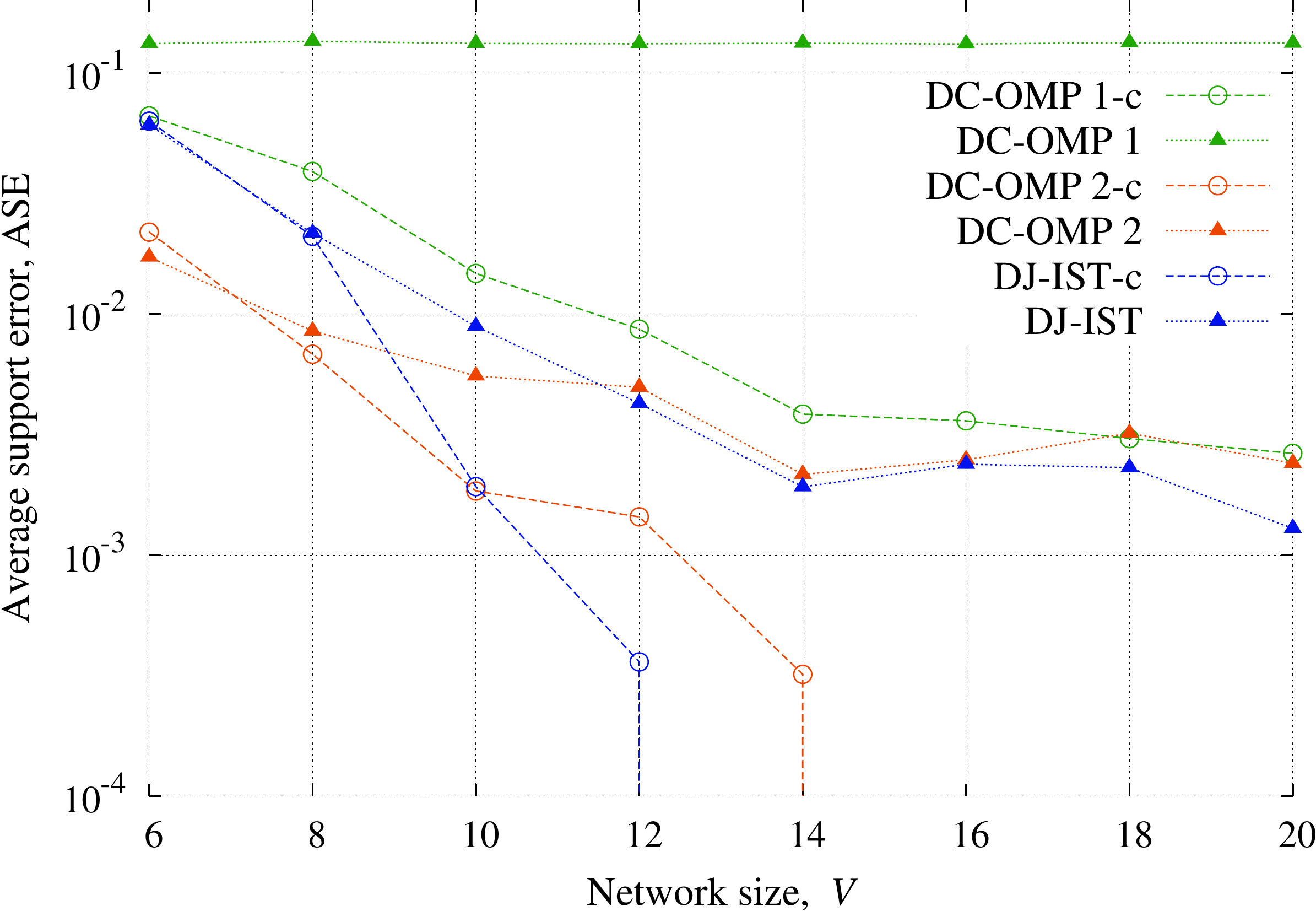}\qquad 
\includegraphics[width=0.95\columnwidth]{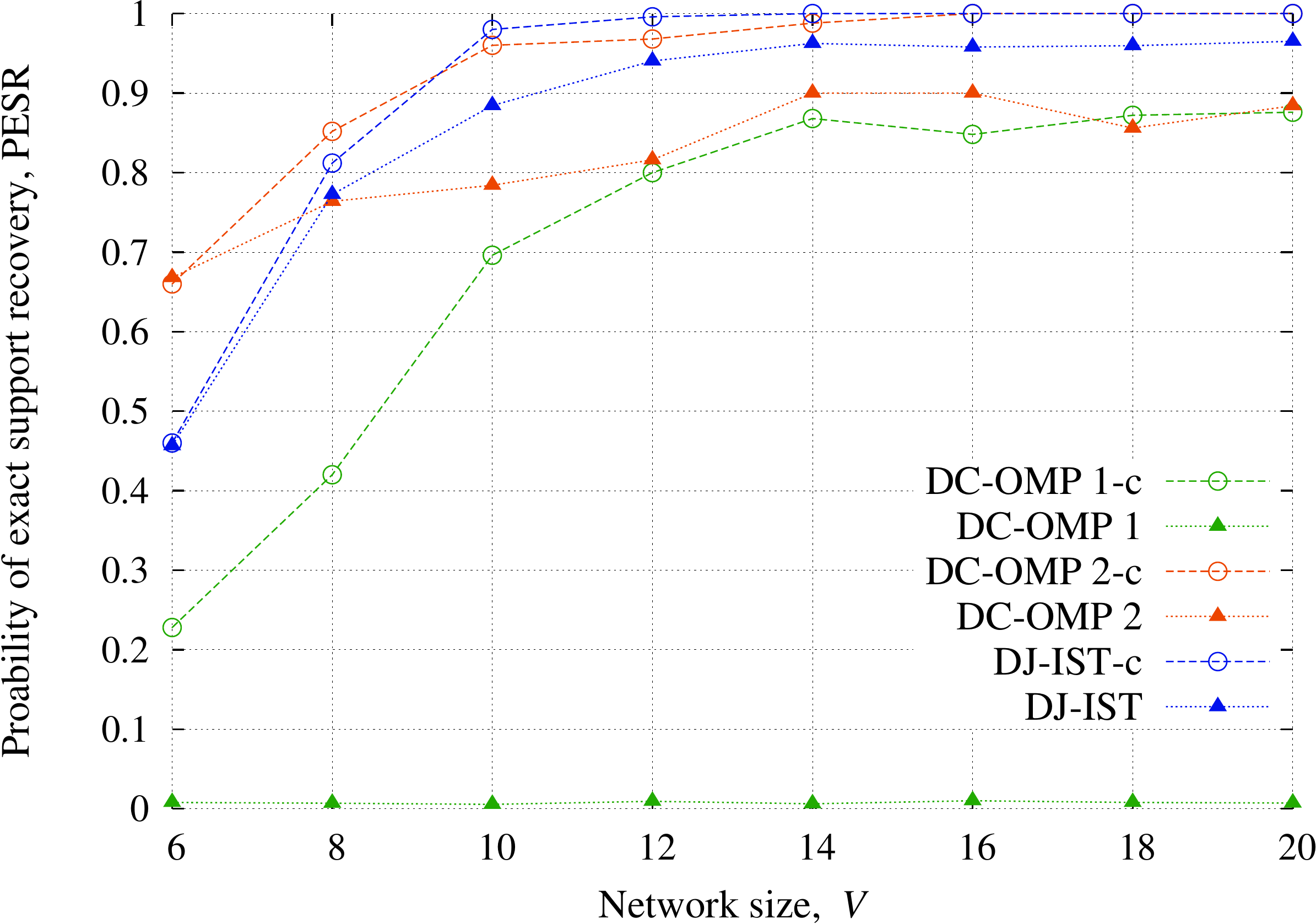}
\caption{ASE (left) and PESR (right) as a function of $\cardV$, $m=18$,  $\lambda=1$, $\alpha=5\times 10^{-4}$, $\beta=1.1$; $\tau=8\times 10^{-3}$ for complete graphs, except for $\cardV\in{6,8}$ where $\tau=3\times 10^{-3}$; $\tau=2\times 10^{-2}$ for 5-regular graphs, except for $\cardV\in{6,8}$ where $\tau=8\times 10^{-3}$} 
\label{fig:numNodes}
\end{figure*}

In Figure \ref{fig:measNode}, we show the ASE and the PESR for a network of $V=10$ nodes, varying of the number of measurements per node $m$ between 4 and 32. We show both the complete graph case (indicated by the postfix {\mbox{'-c'}}) and the regular case with $d=5$ (say, each node has 4 neighbors). The ASE is shown in logarithmic scale: a vertical line indicates the $m$ beyond which the ASE is exactly zero. We immediately notice that DJ-IST (in both complete and non complete regimes) achieves null ASE with a smaller $m$ than all the other methods. Specifically, we observed that $m=22$ is sufficient for DJ-IST to have perfect support detection, while $m=24,28,30$ are necessary respectively for DC-OMP 2-c, DC-OMP 1-c, DC-OMP 2. We further remark that DC-OMP 1 never gets zero in the considered range. 

We also notice that for any considered $m$  DJ-IST  performs better than DC-OMP 1 and less worse than DC-OMP 2 (except for vary small $m$, where DJ-IST is the best). Recalling that DC-OMP 2 always envisages a complete topology (as it exploits global (multihop) communication in the non-complete case), the fact  that DJ-IST-c is very close to DC-OMP 2 is remarkable. Analogous considerations can be done for the PESR curve.

In Figure \ref{fig:numNodes}, we show the ASE and the PESR for fixed $m=18$ and varying $V$. Again, we appreciate that DJ-IST outperforms DC-OMP 1, while the PESR of DJ-IST is  better than that of DC-OMP 2 in the non-complete regime, for large networks.

We remark that for non-complete topologies, support agreement among the nodes is not guaranteed; analytical conditions to get consensus will be subject of future research. However, if necessary,  a consensus algorithm can be run after our procedure to obtain the same support over all the network. 

\subsection{Signal estimation performance}
In addition to the support recovery analysis, we report some observations about the signal estimation accuracy of DJ-IST. In fact, as already remarked, DJ-IST, as a difference from  \cite{sun14, wima14}, performs both support and signal estimation. 

\begin{figure*}
\centering
\includegraphics[width=0.95\columnwidth]{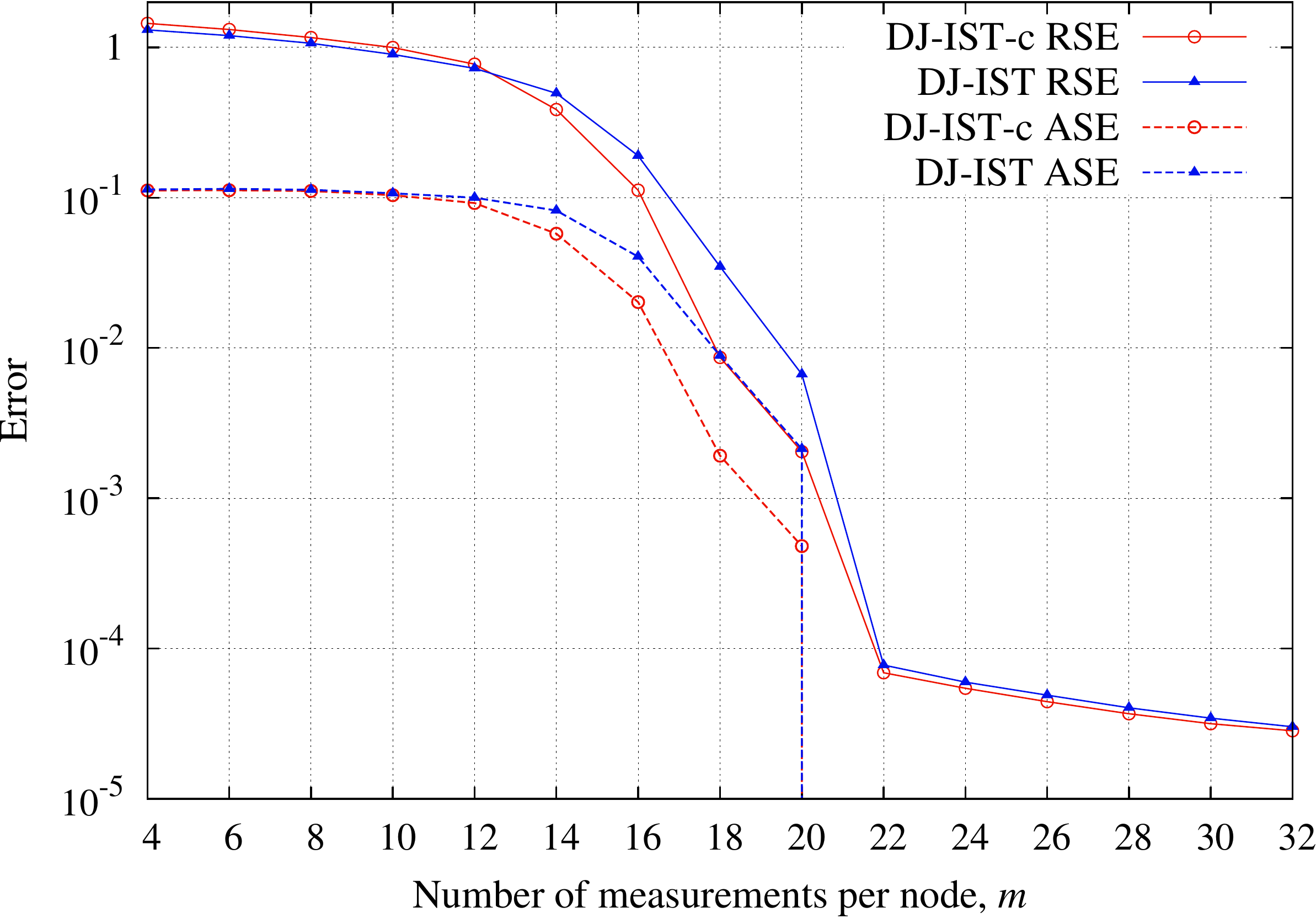}\qquad
\includegraphics[width=0.95\columnwidth]{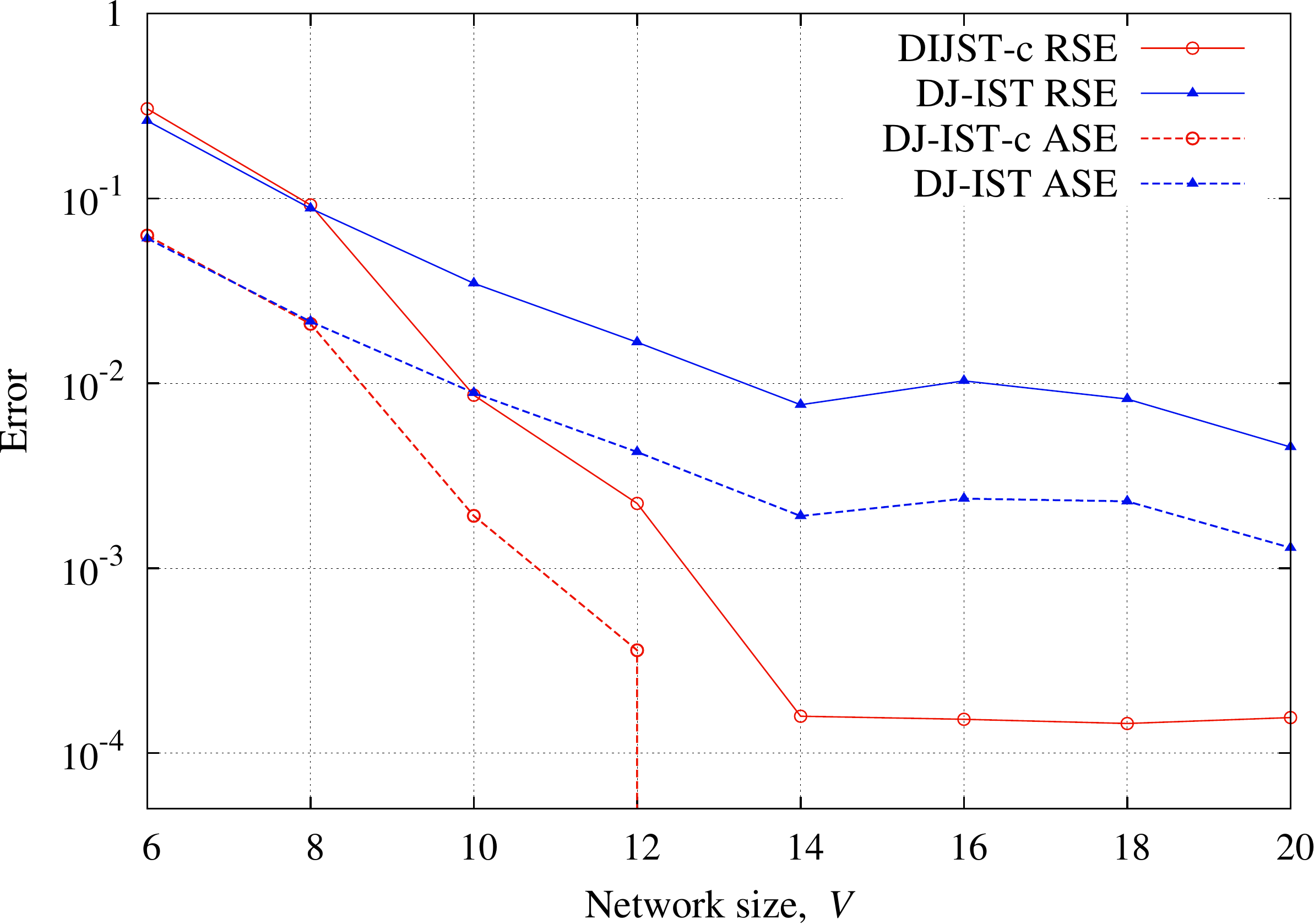} 
\caption{RSE and ASE as a function of $m$ (left) and $V$ (right).}\label{fig:RSE}
\end{figure*}
In Figure \ref{fig:RSE} we depict the mean relative square error (RSE) which we define as 
\begin{equation}
\text{RSE}=  \frac{\sum_{v\in\V} \left\|x^{\star}_v - \widehat{x}_v  \right\|_2^2}{\sum_{v\in\V} \left\|x^{\star}_v  \right\|_2^2}.
\end{equation}
The used parameters  are the ones used in the experiments presented in the previous paragraph, and  RSE and ASE are shown as   functions of  $m$ (left) and $V$ (right). As we are adopting a logarithmic scale, we visualize a vertical line when the ASE goes to zero. In these graphs, we can appreciate that the RSE follows the behavior of the ASE. A small bias occurs in the RSE when the ASE is null, which is expected due to our Lasso approach. The reweighting method reduces the Lasso bias, but does not totally remove it, even though because of Assumption \ref{cosi_contraggo}: shrinkage is reduced, but actually never removed for non-zero coefficients.

\subsection{Analysis of transmission efficiency}
We now analyze the transmission efficiency of DJ-IST, compared to DC-OMP 1 and 2 \cite{wima14}, in terms of number of transmitted bits over each network link. The range of transmitted bits can be analytically evaluated for all the three algorithms, as we now show. Afterwards, we will present some statistics from numerical simulations.

Let us consider the non-complete graph case, and for simplicity let us assume a $d$-regular topology.
In DJ-IST and DC-OMP 1 only  indices in $\{1,\dots,n\}$ are transmitted, then each index can be encoded with $\lfloor \log_2 n \rfloor +1$ bits. In DC-OMP 1, each node $v\in V$ transmits to its $d-1$ neighbors its candidate for activation, namely, the coefficient it would add to the support; afterwards, basically each coefficient with more than two votes is added to the support. Therefore, at each step a maximum a $\lfloor \frac{d}{2} \rfloor$ coefficient could be added, and to complete the support the minimum possible number of step is $\lceil k/\lfloor \frac{d}{2} \rfloor \rceil$, while the maximum is $k$ (one coefficient at each step; we recall that $k$ has to be exactly known in DC-OMP approach, which is not required for DJ-IST). In conclusion, in DJ-IST the total number of bits transmitted over a link is in the range $\cardV (d-1)(\lfloor \log_2 n \rfloor +1)\left[\lceil k/\lfloor \frac{d}{2} \rfloor \rceil,k\right]$.

In DC-OMP 2, the nodes share with neighbors the correlation vector in $\R^n$; assuming $q$ bits for each real value, this amounts to $\cardV(d-1)qn$ bits per iteration. The nodes use such information to choose their own candidate coefficient, and they broadcast it to all the network, which amounts to $\cardV(\cardV-1)(\lfloor \log_2 n \rfloor +1)$. The voting procedure to build the support is  analogous to DC-OMP 1. Hence,  the total number of transmitted bits is in the range $\cardV \left[(d-1)q +(V-1)(\lfloor \log_2 n \rfloor +1)\right]\left[\lceil k/\lfloor \frac{d}{2} \rfloor \rceil,k\right]$.

Differently from DC-OMP strategies, in DJ-IST all the coefficients start as active and then, hopefully, $n-k$ of them are switched to zero.  
Each $v\in V$ communicates to neighbors the switches for non-zero to zero, and vice versa. If all the nodes remain non-zero, no communications occurs, while the maximum is $2pn \cardV (d-1)$, where $p$ is the maximum number of switches from zero to non-zero discussed in Section \ref{sec:algorithm}\footnote{$2pn$ stands for the worst case in which all the coefficients oscillate as long as can, and then switch off to zero.}. 
\begin{table}
\begin{center}
\caption{Transmitted bits:  ranges for $d$-regular topologies ($r=\lfloor \log_2 n \rfloor +1$)}
\centering
\resizebox{1\columnwidth}{!} {
\begin{tabular}{| l | l | l |}
\hline
\hline
Algorithms  & Min  & Max \\
\hline
DC-OMP 1& $\cardV (d-1)r\lceil k/\lfloor \frac{d}{2} \rfloor \rceil$ & $\cardV (d-1)rk$   \\
DC-OMP 2& $\cardV \left[(d-1)q +(V-1)r\right]\lceil k/\lfloor \frac{d}{2} \rfloor \rceil$ & $\cardV \left[(d-1)q +(V-1)r\right]k$\\
DJ-IST& $0$ & $2pn \cardV (d-1)$  \\
\hline
\end{tabular} 
\label{tab:0}}
\end{center}
\end{table}

We sum up these ranges in Table \ref{tab:0}. Next, in Tables \ref{tab:1} and \ref{tab:2}, we show transmission load statistics  taken from our simulations over regular graphs with degree $d=5$ (250 runs). Real values are assumed to be quantized over $q=16$ bits. %Moreover, we observe that the mean is always closer to the minimum in DJ-IST, and closer to maximum for DC-OMP 1 and 2 in numerical results. This suggest a better mean behavior of DJ-IST.

\begin{table}
\begin{center}
\caption{Transmitted bits: statistics over all the simulations with $n=100$, $k=10$, $\cardV=10$, $m\in\{4,6,8,\dots,32\}$)}
%\centering
%\resizebox{2\columnwidth}{!} {
\begin{tabular}{| l | r | r | r |}
\hline
\hline
Algorithms  & Min  & Max  & Mean\\
\hline
DC-OMP 1& 2520 & 2800  & 2795  \\
DC-OMP 2& 193890 & 387780 & 298590\\
DJ-IST& 29288 & 39508  & 32938  \\
\hline
\end{tabular} 
\label{tab:1}
\end{center}
\end{table}

\begin{table}
\begin{center}
\caption{Transmitted bits: statistics over all the simulations with $n=100$, $k=10$, $\cardV\in\{6,20\}$, $m=18$)}
%\centering
%\resizebox{2\columnwidth}{!} {
\begin{tabular}{| l | r | r | r |}
\hline
\hline
Algorithms $~~\cardV=6$  & Min  & Max  & Mean\\
\hline
DC-OMP 1& 1512 &  1680 &   1673 \\
DC-OMP 2& 193050 &  386100 &  328957\\
DJ-IST & 16828 & 34552  & 21750  \\
\hline
\hline
Algorithms  $~~\cardV=20$  & Min  & Max  & Mean\\
\hline
DC-OMP 1& 4480 &  5600 &  5570\\
DC-OMP 2&261320 &  522640 &  373687\\
DJ-IST & 61236  & 92624  & 69568 \\
\hline
\end{tabular} 
\label{tab:2}
\end{center}
\end{table}
%The number of bits transmitted respectively by DJ-IST, DC-OMP 1, and DC-OMP 2 are of order $10^4$, $10^5$, and $10^3$.
%Finally, we observe that distributed reconstruction requires more transmission overhead than centralized reconstruction at a fusion center.
%...
%a usual drawback of distributed reconstruction is the requirement of many more transmissions with respect to centralized reconstruction at a fusion center. This is not the case of DJ-IST. Given a fusion center,  the nodes have to transmit  the $m$ compressed measurements and the $mn$ sensing matrices to it, which typically are real values quantized over $q$ bits. and again assuming  quantization on $q=16$ bits, we will have $q(m+nm)V>10^4$ transmitted bits, which are of the same order of that in Table \ref{tab:2}. Moreover, one should consider that in the applications the distances among nodes are smaller than distances between nodes and fusion center, that is, transmissions to a fusion centers might be much more expensive.
\subsection{DJ-ADMM}\label{sub:ADMM}
In Section \ref{sub:why_not_admm}, we intuitively explained that replacing the IST step in DJ-IST (Step 7 in  Algorithm \ref{DJ-IST}) with faster Lasso decreasing algorithms is not expected to improve the performance.
%, as quickness does not endorse collaboration. Quickness in fact  may force nodes' decisions and decrease the benefits coming from information sharing. This causes a sort of instability: instead of collaborating and reach an agreement on the support, each node tends to follow its own direction. This generates a sort of network conflict, which cause more support switches and more transmissions.
We now show an example: we replace IST with  ADMM \cite{boy10},
%  a popular iterative algorithm, which  has been proved to converge rapidly to minima of Lasso functionals in centralized \cite{luo12} and distributed problems with sparsity constraints \cite{rfm15, mata14, mot13}. 
The settings are as follows: $\lambda=1$, $\alpha=5\times 10^{-3}$. For each $v\in\V$,  we consider the augmented Lagrangian
\begin{equation}\label{lag}  
 \begin{split}
\mathcal{L}&(x_v,z_v;\mu_v)=\frac{1}{2}\left\|y_v-A_v x_v \right\|_2^2+\lambda\alpha\sum_{i=1}^n w_{v,i}| z_{v,i}|\\&+ \rho \mu_v^{\mathsf{T}}(x_v-z_v)+\rho \left\|x_v-z_v \right\|_2^2
 \end{split}
\end{equation}
where $\rho>0$ (here we fix $\rho=1$), $x_v$, $z_v$, $\mu_v \in \R^n$. Given $\mathrm{d}(\mu_v)=\min_{x_v,z_v}\mathcal{L}(x_v,z_v; \mu_v)$, at each step, ADMM  decreases the functional $\mathcal{L}(x_v,z_v; \mu_v)-2\mathrm{d}(\mu_v)$ \cite[Theorem 3.1]{luo12}. Specifically the ADMM step for Lasso is as follows (see \cite[Section 6.4]{boy10}):
\begin{equation*}
 \begin{split}
x_v(t+1) &=\argmin{x_v}\mathcal{L}(x_v,z_v(t);\mu_v(t))\\ &=(A_v^{\mathsf{T}}A_v+\rho I)^{-1}[A_v^{\mathsf{T}}y_v+\rho (z_v(t)-\mu_v(t))]\\
z_v(t+1)&=\argmin{z_v}\mathcal{L}(x_v(t),z_v;\mu_v(t))\\&=\soft_{\lambda\alpha w_v(t) / \rho}[x_v(t+1)+\mu_v(t)].\\
\mu_v(t+1)&=\mu_v(t)+x_v(t+1)-z_v(t+1).
 \end{split}
\end{equation*}
We name DJ-ADMM the algorithm that we obtain by replacing IST with ADMM in DJ-IST, with the usual forced stopping of the null components above a switch threshold $p$. In our simulations, we observed that no more than 5 switches from zero to non-zero occurred using DJ-ADMM, and  as for DJ-IST, in the practice we did not set $p$ in advance.

\begin{figure*}[ht]
\centering
\includegraphics[width=0.65\columnwidth]{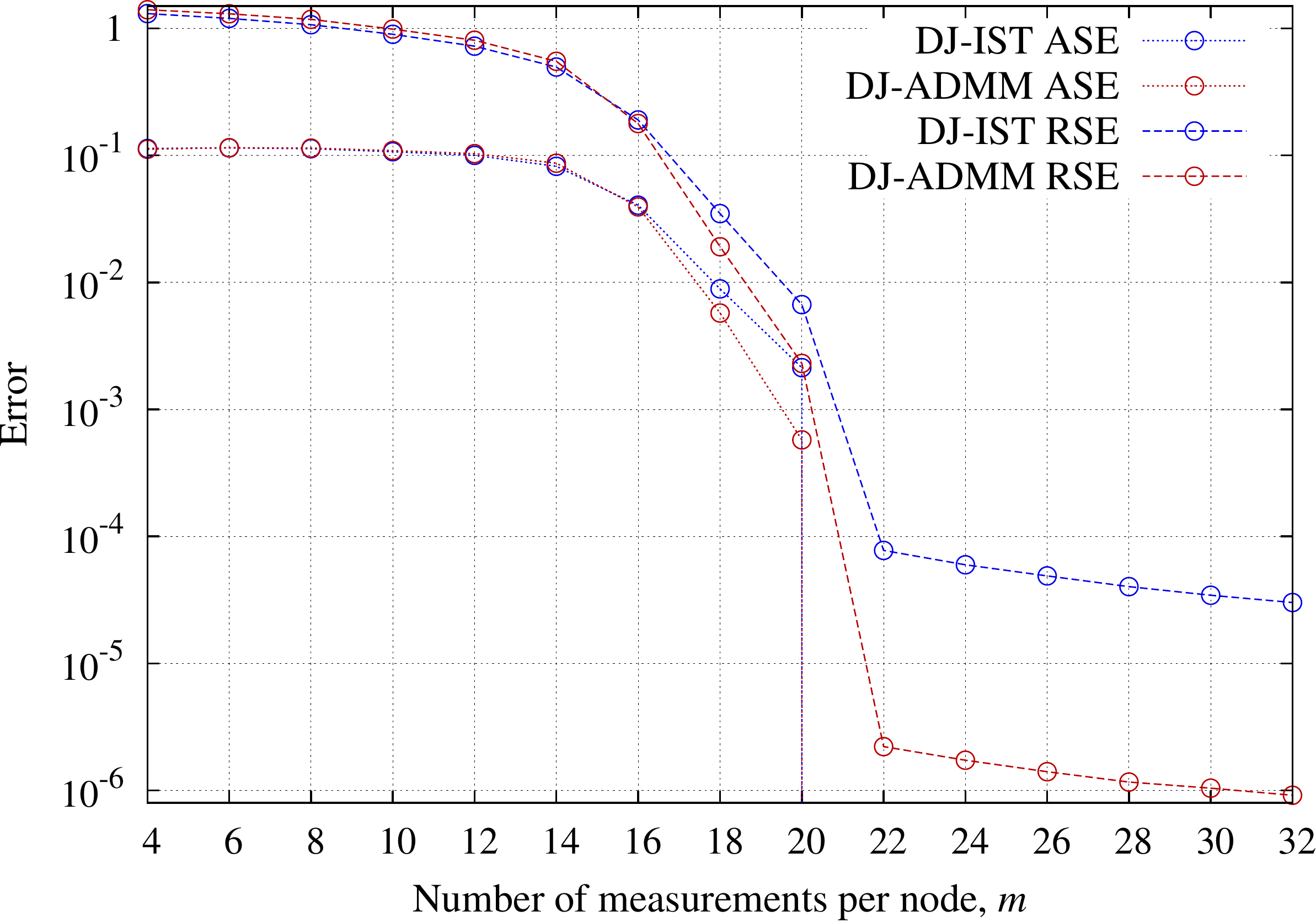}\quad
\includegraphics[width=0.65\columnwidth]{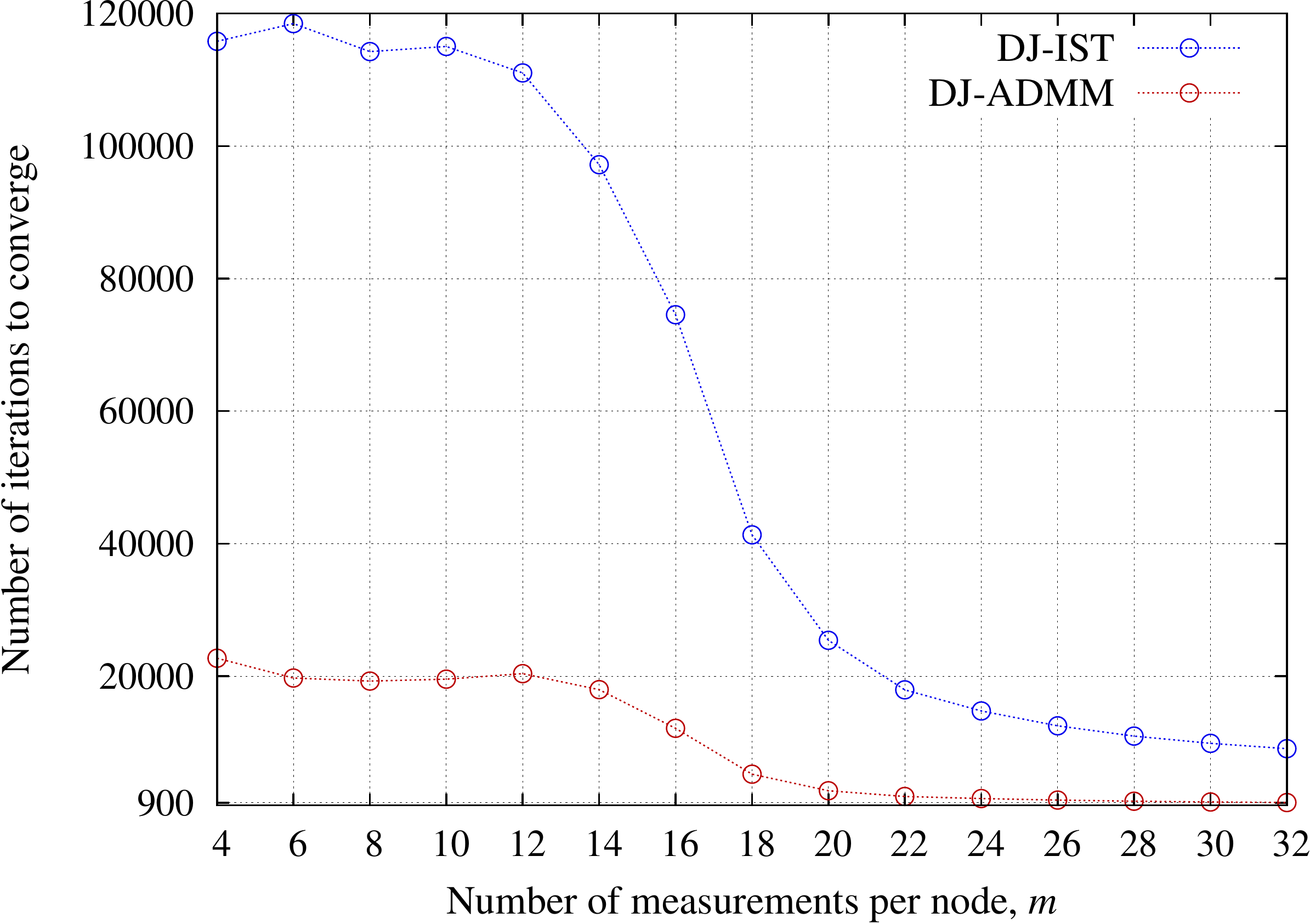} \quad
\includegraphics[width=0.65\columnwidth]{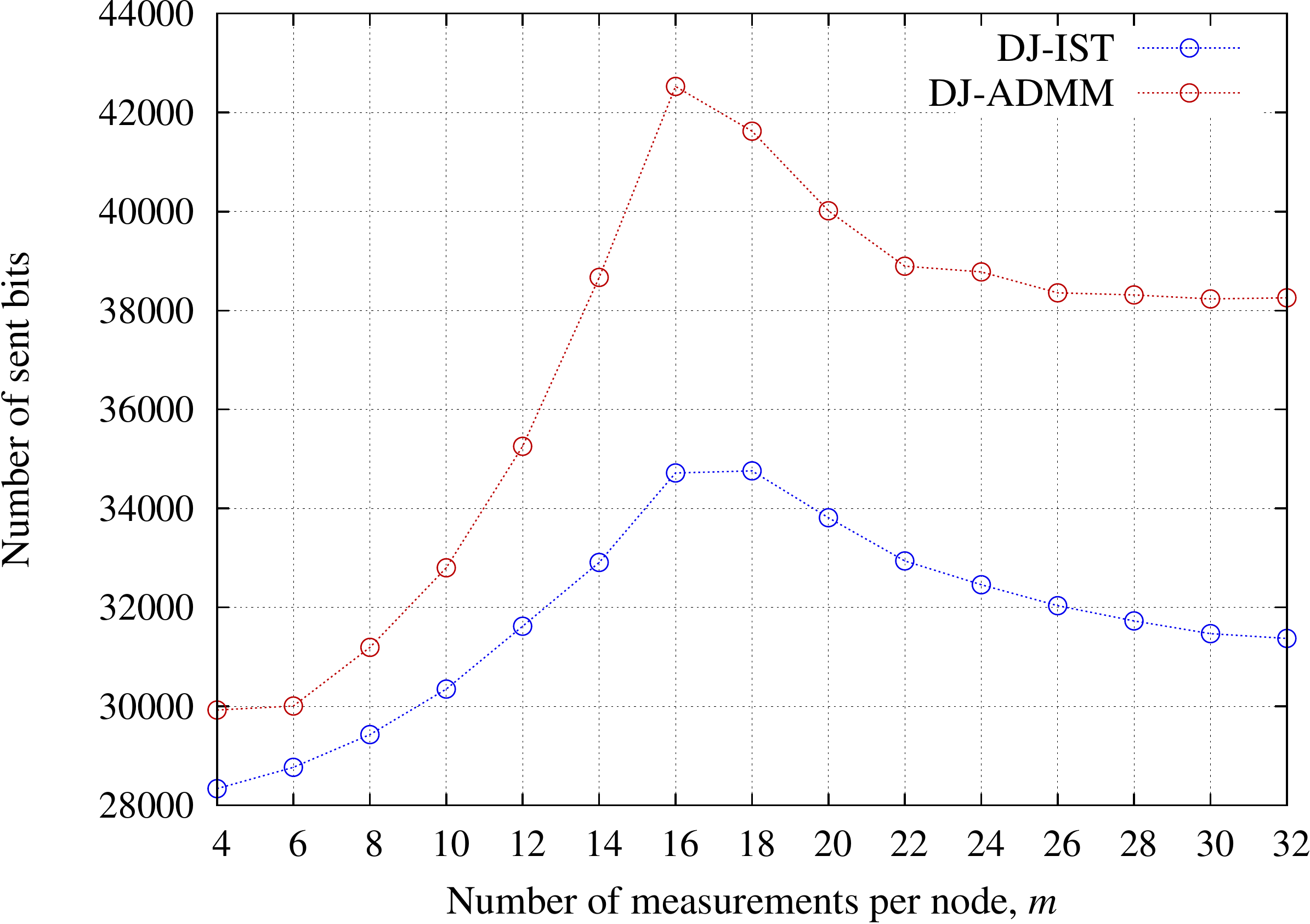} 
\caption{DJ-IST vs DJ-ADMM: ASE, RSE, number of iterations, and sent bits.}\label{fig:DJADMM}
\end{figure*}
In Figure \ref{fig:DJADMM} we compare DJ-IST and DJ-ADMM for varying  $m$, averaged over 250 runs. The setting is the one described in Section  \ref{sub:simulations_setting}, with regular topology with degree 5. First, we show that the support reconstruction accuracy, evaluated in terms of ASE, is very similar. When the support is exactly recovered, the RSE of DJ-ADMM achieves $10^{-6}$, while DJ-IST is around $10^{-5}$, due to the bias that can be evaluated from \eqref{puntofisso}.

We further observe that DJ-ADMM is much faster in terms of number of iterations (second graph of Figure \ref{fig:DJADMM}), but  requires a larger  number of bit transmissions (third graph). As already explained, this is expected as ADMM forces a faster decrease of the Lasso, which may produce conflicts with the information gathered from the network; the behavior of the single node is then too aggressive, which causes more switches, hence more transmissions, if compared to DJ-IST. However, the number of transmissions of DJ-ADMM is still of the order of DJ-IST. This makes DJ-ADMM suitable for those cases in which velocity is desired.

Regarding the number of transmitted bits, we remark the peak (for both DJ-IST and DJ-ADMM) for mid values of $m$. The reason is that when few measurements are available, each node has less information to communicate; on the other hand, many measurements allow a faster convergence and less transmissions. Thus, it is in the intermediate case that the network has its most intense activity.

\section{Conclusion}\label{sec:conclusion}
In this paper, we have proposed DJ-IST, a distributed soft thresholding algorithm to recover jointly sparse signals. The shrinkage thresholds  are reweighted at each step, based on information on the support coming from the network. DJ-IST estimates both the support and the non-zero values of the unknown signals. DJ-IST is proved to converge to a minimum of a suitable cost functional with concave penalization. Interestingly, DJ-IST can be interpreted as a distributed reweighted $\ell_1$ minimization algorithm. In terms of support recovery accuracy, DC-OMP 2  is the state-of-the-art method. Numerical simulations show that DJ-IST has a  performance close to DC-OMP 2, but significantly outperforms it in terms of transmission efficiency (namely, number of transmitted bits per link). On the other hand,  DC-OMP 1 is the state-of-the-art method in terms of transmission efficiency, but its performance is shown to be worse than DJ-IST. In conclusion, DJ-IST is an optimal trade-off between recovery performance and energy saving capability, which makes it more suitable than greedy procedures.

The scheme of DJ-IST seems to be applicable to other jointly sparse models, like JSM-1 and JSM-3 \cite{dua05}, that have been recently tackled with distributed algorithms \cite{mata14, mata15}. Moreover, we remark that DJ-IST could be used in case of recovery of a unique common signal \cite{rfm15} to improve the  transmission efficiency \cite{rav14icassp, rav15}:  sharing information  about the support instead of transmitting the whole signal's estimate may dramatically reduce the communication load. These points will be subject of our future work.

%As a difference from distributed greedy pursuit algorithms, like DC-OMP 1 and DC-OMP 2 \cite{wima14}, DJ-IST has the advantage of not requiring any prior knowledge about the sparsity level $k$; moreover, the choice of the fixed parameters (namely, $\lambda$ and $\alpha$) can be done empirically and is not critical.
\bibliographystyle{plain} 
%\bibliography{refs}

\begin{thebibliography}{10}

\bibitem{ant11}
Anestis Antoniadis and Jianqing Fan.
\newblock Regularization of wavelet approximations.
\newblock {\em J. Amer. Statist. Assoc.}, 96(45):939 -- 967, 2011.

\bibitem{ara15}
B.~Aragam and Q.~Zhou.
\newblock Concave penalized estimation of sparse gaussian bayesian networks.
\newblock {\em J. Mach. Learn. Res.}, in press, 2015.

\bibitem{bag13}
A.~M Bagirov, L~Jin, N.~Karmitsa, A.~Al~Nuaimat, and N.~Sultanova.
\newblock Subgradient method for nonconvex nonsmooth optimization.
\newblock {\em J. Optim. Theory Appl.}, 157(2):416--435, 2013.

\bibitem{bar05}
D.~Baron, M.~F. Duarte, M.~B. Wakin, S.~Sarvotham, and R.~G. Baraniuk.
\newblock Distributed compressive sensing of jointly sparse signals.
\newblock In {\em Asilomar Conf. Signals, Sys., Comput.}, pages 1537--1541,
  2005.

\bibitem{baz10}
J.~A. Bazerque and G.~B. Giannakis.
\newblock Distributed spectrum sensing for cognitive radio networks by
  exploiting sparsity.
\newblock {\em IEEE Trans. Signal Process.}, 58(3):1847--1862, 2010.

\bibitem{bec09}
A.~Beck and M.~Teboulle.
\newblock A fast iterative shrinkage-thresholding algorithm for linear inverse
  problems.
\newblock {\em SIAM J. Imaging Sci.}, 2(1):183--202, 2009.

\bibitem{bla14}
J.D. Blanchard, M.~Cermak, D.~Hanle, and Yirong Jing.
\newblock Greedy algorithms for joint sparse recovery.
\newblock {\em IEEE Trans. Signal Process.}, 62(7):1694--1704, April 2014.

\bibitem{blu08}
T.~Blumensath and M.~E. Davies.
\newblock Iterative thresholding for sparse approximations.
\newblock {\em J. Fourier Anal. Appl.}, 14(5-6):629 -- 654, 2008.

\bibitem{bol14}
J{\'e}r{\^o}me Bolte, Shoham Sabach, and Marc Teboulle.
\newblock Proximal alternating linearized minimization for nonconvex and
  nonsmooth problems.
\newblock {\em Math. Program., Ser. A}, 146(1):459--494, 2013.

\bibitem{boy10}
S.~Boyd, N.~Parikh, E.~Chu, B.~Peleato, and J.~Eckstein.
\newblock Distributed optimization and statistical learning via the alternating
  direction method of multipliers.
\newblock {\em Found. Trends Mach. Learn.}, 3(1):1 -- 122, 2010.

\bibitem{bur05}
J.~V. Burke, A.~S. Lewis, and M.~L. Overton.
\newblock A robust gradient sampling algorithm for nonsmooth, nonconvex
  optimization.
\newblock {\em SIAM J. Optim.}, 15(3):751--779, 2005.

\bibitem{can08rew}
E.~J. Cand\`es, M.~B. Wakin, and S.P. Boyd.
\newblock Enhancing sparsity by reweighted $\ell_1$ minimization.
\newblock {\em Journ. Fourier Anal. Appl.}, 14(5-6):877--905, 2008.

\bibitem{che06}
J.~Chen and X.~Huo.
\newblock Theoretical results on sparse representations of multiple-measurement
  vectors.
\newblock {\em IEEE Trans. Signal Process.}, 54(12):4634--4643, Dec 2006.

\bibitem{che12}
X.~Chen.
\newblock Smoothing methods for nonsmooth, nonconvex minimization.
\newblock {\em Math. Program. Ser. B}, 134(1):71--99, 2012.

\bibitem{che14}
Xiaojun Chen and Weijun Zhou.
\newblock Convergence of the reweighted $\ell_1$ minimization algorithm for
  $\ell_2$-$\ell_p$ minimization.
\newblock {\em Computational Optimization and Applications}, 59(1-2):47--61,
  2014.

\bibitem{dav12}
M.E. Davies and Y.C. Eldar.
\newblock Rank awareness in joint sparse recovery.
\newblock {\em IEEE Trans. Inf. Theory}, 58(2):1135--1146, Feb 2012.

\bibitem{don06}
D.~L. Donoho.
\newblock Compressed sensing.
\newblock {\em IEEE Trans. Inf. Theory}, 52:1289 -- 1306, 2006.

\bibitem{dua05}
M.~F. Duarte, S.~Sarvotham, D.~Baron, M.~B. Wakin, and R.~G. Baraniuk.
\newblock Distributed compressed sensing of jointly sparse signals.
\newblock In {\em Proc. 39th Asilomar Conf. Signals, Systems and Computers,
  Pacific Grove, CA}, 2005.

\bibitem{fan01_pioneer}
Jianqing Fan and Runze Li.
\newblock Variable selection via nonconcave penalized likelihood and its oracle
  properties.
\newblock {\em J. Amer. Statist. Assoc.}, 96(456):1348 -- 1360, 2001.

\bibitem{fan11}
Jianqing Fan and Jinchi Lv.
\newblock Nonconcave penalized likelihood with np-dimensionality.
\newblock {\em IEEE Trans. Inf. Theory}, 57(8):5467 -- 5484, Aug 2011.

\bibitem{fan14}
Jianqing Fan, Lingzhou Xue, and Hui Zou.
\newblock Strong oracle optimality of folded concave penalized estimation.
\newblock {\em Ann. Statist.}, 42(3):819 -- 849, 06 2014.

\bibitem{faz03}
Maryam Fazel, Haitham Hindi, and Stephen~P Boyd.
\newblock Log-det heuristic for matrix rank minimization with applications to
  hankel and euclidean distance matrices.
\newblock In {\em American Control Conference, 2003. Proceedings of the 2003},
  volume~3, pages 2156 -- 2162. IEEE, 2003.

\bibitem{for10}
M.~Fornasier.
\newblock Numerical methods for sparse recovery.
\newblock In M.~Fornasier, editor, {\em Theoretical Foundations and Numerical
  Methods for Sparse Recovery}, pages 93--200. Radon Series Comp. Appl. Math.,
  de Gruyter, 2010.

\bibitem{fox14}
S.~M. Fosson, J.~Matamoros, C.~Ant\'{o}n-Haro, and E.~Magli.
\newblock Distributed support detection of jointly sparse signals.
\newblock In {\em IEEE International Conference on Acoustics, Speech and Signal
  Processing (ICASSP)}, pages 6434--6438. IEEE, 2014.

\bibitem{fou11}
S.~Foucart.
\newblock Hard thresholding pursuit: An algorithm for compressive sensing.
\newblock {\em SIAM J. Numer. Anal.}, 49(6):2543--2563, 2011.

\bibitem{gas09}
G.~Gasso, A.~Rakotomamonjy, and S.~Canu.
\newblock Recovering sparse signals with a certain family of nonconvex
  penalties and dc programming.
\newblock {\em IEEE Trans. Signal Process.}, 57(12):4686 -- 4698, 2009.

\bibitem{luo12}
M.~Hong and Z.~Luo.
\newblock On the linear convergence of the alternating direction method of
  multipliers.
\newblock {\em arXiv preprint arXiv:1208.3922}, 2012.

\bibitem{kyr11}
A.~Kyrillidis and V.~Cevher.
\newblock Recipes on hard thresholding methods.
\newblock In {\em IEEE International Workshop on Computational Advances in
  Multi-Sensor Adaptive Processing (CAMSAP)}, pages 353--356, 2011.

\bibitem{kyr12}
A.~Kyrillidis and V.~Cevher.
\newblock Combinatorial selection and least absolute shrinkage via the clash
  algorithm.
\newblock In {\em IEEE International Symposium on Information Theory
  Proceedings (ISIT)}, pages 2216--2220, 2012.

\bibitem{wei13}
Weifeng Li, Yicong Zhou, N.~Poh, Fei Zhou, and Qingmin Liao.
\newblock Feature denoising using joint sparse representation for in-car speech
  recognition.
\newblock {\em IEEE Sig. Proc. Letters}, 20(7):681--684, July 2013.

\bibitem{lin11}
Q.~Ling and Z.~Tian.
\newblock Decentralized support detection of multiple measurement vectors with
  joint sparsity.
\newblock In {\em IEEE ICASSP}, pages 2996--2999, 2011.

\bibitem{mata14}
J.~Matamoros, S.~M. Fosson, E.~Magli, and C.~Ant\'{o}n-Haro.
\newblock Distributed {ADMM} for in-network reconstruction of sparse signals
  with innovations.
\newblock In {\em IEEE Global Conference on Signal and Information Processing
  (GlobalSIP)}, pages 429 -- 433, 2014.

\bibitem{mata15}
J.~Matamoros, S.~M. Fosson, E.~Magli, and C.~Ant\'{o}n-Haro.
\newblock Distributed {ADMM} for in-network reconstruction of sparse signals
  with innovations.
\newblock {\em IEEE Trans. Signal Inf. Process. Netw.}, 1(4):225 -- 234, 2015.

\bibitem{rav14icassp}
C.~Ravazzi, S.~M. Fosson, and E.~Magli.
\newblock Energy-saving gossip algorithm for compressed sensing in multi-agent
  systems.
\newblock In {\em Proc. of IEEE ICASSP}, pages 5060 -- 5064, 2014.

\bibitem{rfm15}
C.~Ravazzi, S.~M. Fosson, and E.~Magli.
\newblock Distributed iterative thresholding for $\ell_0$/$\ell_1$-regularized
  linear inverse problems.
\newblock {\em IEEE Trans. Inf. Theory}, 61(4):2081 -- 2100, 2015.

\bibitem{rav15}
C.~Ravazzi, S.~M. Fosson, and E.~Magli.
\newblock Randomized algorithms for distributed nonlinear optimization under
  sparsity constraints.
\newblock {\em IEEE Trans. Signal Process. (to appear)}, 2015.

\bibitem{she14}
S.~Shekhar, V.M. Patel, N.M. Nasrabadi, and R.~Chellappa.
\newblock Joint sparse representation for robust multimodal biometrics
  recognition.
\newblock {\em IEEE Trans. Patt. Ana. \& Mach. Intel.}, 36(1):113--126, 2014.

\bibitem{sun14}
Dennis Sundman, Saikat Chatterjee, and Mikael Skoglund.
\newblock Distributed greedy pursuit algorithms.
\newblock {\em Signal Processing}, 105:298--315, 2014.

\bibitem{tib13}
Ryan~J. Tibshirani.
\newblock {The {L}asso problem and uniqueness}.
\newblock {\em Electronic Journal of Statistics}, 7:1456--1490, 2013.

\bibitem{wai09}
Martin~J. Wainwright.
\newblock Sharp thresholds for high-dimensional and noisy sparsity recovery
  using {$\ell_1$}-constrained quadratic programming ({L}asso).
\newblock {\em IEEE Trans. Inf. Theory}, 55(5):2183--2202, 2009.

\bibitem{wima14}
T.~Wimalajeewa and P.K. Varshney.
\newblock {OMP} based joint sparsity pattern recovery under communication
  constraints.
\newblock {\em IEEE Trans. Signal Process.}, 62(19):5059--5072, Oct 2014.

\bibitem{wip10}
D.~Wipf and S.~Nagarajan.
\newblock Iterative reweighted $\ell _1$ and $\ell _2$ methods for finding
  sparse solutions.
\newblock {\em IEEE J. Sel. Top. Signal Process.}, 4(2):317--329, April 2010.

\bibitem{yan11}
J.~Yang and Y.~Zhang.
\newblock Alternating direction algorithms for $\ell_1$-problems in compressive
  sensing.
\newblock {\em SIAM J. Sci. Comp.}, 33(1):250--278, 2011.

\bibitem{nan11}
Nannan Yu, Tianshuang Qiu, Feng Bi, and Aiqi Wang.
\newblock Image features extraction and fusion based on joint sparse
  representation.
\newblock {\em IEEE J. Sel. Top. Sign. Proces.}, 5(5):1074--1082, Sept 2011.

\bibitem{yua06}
Ming Yuan and Yi~Lin.
\newblock Model selection and estimation in regression with grouped variables.
\newblock {\em Journal of the Royal Statistical Society: Series B (Statistical
  Methodology)}, 68(1):49--67, 2006.

\bibitem{yua12}
X.-T. Yuan, X.~Liu, and S.~Yan.
\newblock Visual classification with multitask joint sparse representation.
\newblock {\em IEEE Trans. Image Process.}, 21(10):4349--4360, Oct 2012.

\bibitem{zen11}
F.~Zeng, C.~Li, and Z.~Tian.
\newblock Distributed compressive spectrum sensing in cooperative multihop
  cognitive networks.
\newblock {\em IEEE J. Sel. Top. Sign. Proces.}, 5(1):37--48, 2011.

\bibitem{zha10MCP}
Cun-Hui Zhang.
\newblock Nearly unbiased variable selection under minimax concave penalty.
\newblock {\em Ann. Statist.}, 38(2):894 -- 942, 2010.

\bibitem{zha12}
Cun-Hui Zhang and Tong Zhang.
\newblock A general theory of concave regularization for high-dimensional
  sparse estimation problems.
\newblock {\em Statist. Sci.}, 27(4):576 -- 593, 11 2012.

\bibitem{zha06}
Peng Zhao and Bin Yu.
\newblock On model selection consistency of {L}asso.
\newblock {\em J. Mach. Learn. Res.}, 7:2541 -- 2563, 2006.

\bibitem{zou08_LLA}
Hui Zou and Runze Li.
\newblock One-step sparse estimates in nonconcave penalized likelihood models.
\newblock {\em Annals of Statistics}, 36(4):1509, 2008.

\end{thebibliography}

\end{document}